\documentclass[a4paper]{article}

\usepackage{amsfonts, latexsym, amssymb, amsthm, amsmath, mathrsfs, esint, graphicx, subcaption}
\usepackage[shortlabels]{enumitem}

\newcommand{\R}{\mathbb{R}}
\newcommand{\N}{\mathbb{N}}

\newcommand{\K}{\mathcal{K}}
\newcommand{\id}{\mathrm{id}}
\newcommand{\loc}{\mathrm{loc}}

\DeclareMathOperator*{\esssup}{ess \, sup}
\DeclareMathOperator{\proj}{proj}

\newcommand{\blank}{{\mkern 2mu\cdot\mkern 2mu}}

\newcommand{\set}[2]{\left\{ #1 \colon #2 \right\}}

\newtheorem{theorem}{Theorem}
\newtheorem{lemma}[theorem]{Lemma}
\newtheorem{proposition}[theorem]{Proposition}
\theoremstyle{definition}
\newtheorem{definition}[theorem]{Definition}
\newtheorem{example}[theorem]{Example}

\begin{document}

\title{Structure and classification results for the $\infty$-elastica problem}

\author{Roger Moser\footnote{Department of Mathematical Sciences,
University of Bath, Bath BA2 7AY, UK.
E-mail: r.moser@bath.ac.uk}}

\maketitle

\begin{abstract}
Consider the following variational problem: among all
curves in $\R^n$ of fixed length with prescribed end points and prescribed tangents at the end points,
minimise the $L^\infty$-norm of the curvature.
We show that the solutions of this problem, and of a generalised version,
are characterised by a system of differential equations.
Furthermore, we have a lot of information about the structure of solutions,
which allows a classification.
\end{abstract}

\section{Introduction}

Variational problems involving the curvature of a curve
$\Gamma \subseteq \R^n$ have a long history. This is true especially for
the Euler elastica problem, which is to minimise the quantity
\[
\int_\Gamma \kappa^2 \, ds,
\]
where $\kappa$ is the curvature of $\Gamma$.
This functional may be regarded as a model for the stored elastic energy of a thin rod and
its theory can be traced back to Jacob and Daniel Bernoulli and to Euler \cite{Euler:1744,Oldfather-Ellis-Brown:33}, but
the problem has also been studied in more modern times
\cite{Bryant-Griffiths:86, Langer-Singer:84_2, Langer-Singer:84_1, Linner:98}.
An obvious generalisation is the $p$-elastica problem for
$p \in [1, \infty)$, which corresponds to the quantity $\int_\Gamma \kappa^p \, ds$.
This functional has been proposed for applications in image processing \cite{Masnou-Morel:98} and
has also been studied in its own right \cite{Huang:04, Ferone-Kawohl-Nitsch:18}.

While the step from elastica to $p$-elastica amounts to replacing an $L^2$-norm by an $L^p$-norm,
in this paper we consider curves minimising the $L^\infty$-norm of the
curvature. Thus, roughly speaking, we wish to minimise the maximum curvature.
This quantity may not directly appear as the energy of a physical problem, but
questions related to it are of fundamental geometric interest and may appear in design
problems as well. In effect we ask, how much does a curve have to be bent in order to
satisfy certain constraints? We consider constraints in the form of a fixed length 
combined with boundary conditions,
but other types are conceivable as well and may admit a similar theory.

To my knowledge, the $\infty$-elastica problem has not been studied before.
The step from $p < \infty$ to $p = \infty$ changes the nature of the problem significantly.
In particular, we have a functional that is not differentiable in any
meaningful sense and the usual steps to find an Euler-Lagrange equation do no longer work.
While we still have the notion of a minimiser, there is no obvious way to define critical points.
In this paper, we propose another concept instead, derive a
system of equations that can be thought of as Euler-Lagrange equations, and finally
analyse and classify the solutions.

In addition to the standard $L^\infty$-norm, the theory in this paper allows a weighted version as well.
We therefore consider the following set-up of the problem. Let $n \in \N$ with $n \ge 2$.
We fix a number $\ell > 0$, which is
the prescribed length of the curves considered. We also fix a
weight function $\alpha \colon [0, \ell] \to (0, \infty)$,
which should be of bounded variation and such that $1/\alpha$ is bounded.
We represent curves in $\R^n$ by parametrisations $\gamma \colon [0, \ell] \to \R^n$ by arc length for the moment,
so we assume that $|\gamma'| \equiv 1$ in $[0, \ell]$. The curvature is then $\kappa = |\gamma''|$.
As we wish to consider its (weighted) $L^\infty$-norm, we assume that $\gamma$ belongs to
the Sobolev space $W^{2, \infty}((0, \ell); \R^n)$ and we define the functional
\[
\K_\alpha(\gamma) = \esssup_{[0, \ell]} \alpha |\gamma''|.
\]

We consider a problem for curves with prescribed end points and prescribed tangent vectors
at these end points. Thus for fixed $a_1, a_2 \in \R^n$ and fixed $T_1, T_2 \in S^{n - 1} = \set{x \in \R^n}{|x| = 1}$,
we require that
\begin{equation} \label{eqn:boundary_conditions}
\gamma(0) = a_1, \quad \gamma(\ell) = a_2, \quad \gamma'(0) = T_1, \quad \text{and} \quad \gamma'(\ell) = T_2.
\end{equation}
Let $\mathcal{G}$ denote the set of all $\gamma \in W^{2, \infty}((0, \ell); \R^n)$ with $|\gamma'| \equiv 1$
in $[0, \ell]$ satisfying \eqref{eqn:boundary_conditions}.
We are particularly interested in minimisers of $\K_\alpha$ in $\mathcal{G}$,
but the observations in this paper suggest to consider the following weaker notion as well.

\begin{definition}[$\infty$-elastica] \label{def:infty-elastica}
Suppose that $\gamma \in \mathcal{G}$. We say that $\gamma$ is an \emph{$\infty$-elastica} if 
there exists $M \in \R$ such that for every $\tilde{\gamma} \in \mathcal{G}$, the inequality
\[
\K_\alpha(\gamma) \le \K_\alpha(\tilde{\gamma}) + \frac{M}{2} \int_0^\ell |\tilde{\gamma}' - \gamma'|^2 \, ds
\]
holds true.
\end{definition}

It turns out that this condition is equivalent to a system of differential equations.
Connections between a variational problem and
differential equations are of course quite common, but for a functional that is not differentiable, such a
strong correspondence is surprising.
In order to write down the system concisely, we introduce some notation:
if $V, W \in \R^n$, then $\proj_{V, W}^\perp$ denotes the orthogonal projection onto the
orthogonal complement of the linear subspace of $\R^n$ spanned by $V$ and $W$.

\begin{theorem}[Characterisation by differential equations] \label{thm:DE}
Suppose that $\gamma \in \mathcal{G}$, and let $T = \gamma'$ and $k = \K_\alpha(\gamma)$.
Then $\gamma$ is an $\infty$-elastica if, and only if, there exist $\lambda \in S^{n - 1}$
and $g \in W^{1, \infty}(0, \ell) \setminus \{0\}$ with $g \ge 0$ such that the equations
\begin{align}
g((\alpha T')' + k^2T/\alpha) & = k^2 \proj_{T, T'}^\perp(\lambda), \label{eqn:ODE1} \\
g' & = \alpha \lambda \cdot T' \label{eqn:ODE2}
\end{align}
hold weakly in $(0, \ell)$.
\end{theorem}

It is clear how to interpret weak solutions of \eqref{eqn:ODE2}. In order to make sense of \eqref{eqn:ODE1},
we use that fact that $g$, being in $W^{1, \infty}(0, \ell)$, has a weak derivative. Thus \eqref{eqn:ODE1}
is satisfied weakly if
\[
\int_0^\ell \left(g \alpha T' \cdot \xi' + g'\alpha T' \cdot \xi - gk^2 \alpha^{-1} T \cdot \xi + k^2 \proj_{T, T'}^\perp(\lambda) \cdot \xi\right) \, ds = 0
\]
for all $\xi \in C_0^\infty((0, \ell); \R^n)$.

If we add another condition, we obtain a
criterion for minimisers of $\K_\alpha$, too.

\begin{theorem}[Sufficient condition for minimisers] \label{thm:minimiser}
Let $\gamma \in \mathcal{G}$ and $T = \gamma'$.
Suppose that there exist $\lambda \in S^{n - 1}$ and $g \in W^{1, \infty}(0, \ell) \setminus \{0\}$ such that
\eqref{eqn:ODE1} and \eqref{eqn:ODE2} are satisfied weakly
in $(0, \ell)$, and such that
$0 \le g \le -\alpha \lambda \cdot T$ in $[0, \ell]$. Then $\gamma$ minimises $\K_\alpha$ subject
to the boundary conditions \eqref{eqn:boundary_conditions}.
\end{theorem}

This condition is sufficient but not necessary, as shown in Example \ref{ex:arc} below.

It is worthwhile to consider the case $\alpha \equiv 1$ separately, as the
system \eqref{eqn:ODE1}, \eqref{eqn:ODE2} can then be written as a single equation, albeit with an additional parameter.
This is because in this case, the right-hand side of \eqref{eqn:ODE2} is the derivative
of $\lambda \cdot T$ and the equation implies that there exists $\eta \in \R$ such that
$g = \lambda \cdot T - \eta$. Thus
\begin{equation} \label{eqn:ODE_alpha=1}
T'' + k^2T = \frac{k^2 \proj_{T, T'}^\perp(\lambda)}{\lambda \cdot T - \eta},
\end{equation}
at least where $\lambda \cdot T \neq \eta$.
The left-hand side is a geometric quantity related to the torsion of the corresponding curve if $n = 3$.
Indeed, it can be seen, with arguments as in Proposition \ref{prop:equivalence} below, that the torsion is
$\pm k^{-1} |T'' + k^2 T|$.

Analysing the system \eqref{eqn:ODE1}, \eqref{eqn:ODE2}, we obtain
good information about the structure of $\infty$-elasticas as well,
which allows a classification.

\begin{theorem}[Structure and classification] \label{thm:structure}
Suppose that $\gamma \in \mathcal{G}$ and let $T = \gamma'$ and $k = \K_\alpha(\gamma)$.
Then $\gamma$ is an $\infty$-elastica if, and only if,
there exists $\lambda \in S^{n - 1}$ such that at least one
of the following statements holds true.
\begin{enumerate}[(i)]
\item \label{item:2D}
There exists a line $\mathcal{L} \subseteq \R^n$ parallel to $\lambda$ and there exist
finitely many intervals $J_1, \dotsc, J_N \subseteq [0, \ell]$, pairwise disjoint and open relative to
$[0, \ell]$, such that $\gamma^{-1}(\mathcal{L}) = [0, \ell] \setminus \bigcup_{i = 1}^N J_i$
and such that for $i = 1, \dotsc, N$,
\begin{itemize}
\item 
$\gamma(\overline{J}_i) \cup \mathcal{L}$ is contained in a plane,
\item $\alpha \gamma''$ is continuous with $\alpha |\gamma''| \equiv k$ in $J_i$, and
\item
for any $s_0 \in \overline{I}_i \setminus I_i$, there exists $\delta > 0$ such that
$\lambda \cdot \gamma'' > 0$ in $(s_0, s_0 + \delta) \cap I_i$ and
$\lambda \cdot \gamma'' < 0$ in $(s_0 - \delta, s_0) \cap I_i$.
\end{itemize}

\item \label{item:3D}
There is a three-dimensional affine subspace of $\R^n$ that contains $\gamma([0, \ell])$.
Furthermore, $\alpha \gamma'' \in W^{1, \infty}((0, \ell);\R^n)$ with $\alpha |\gamma''| \equiv k$ and there exists $g \in W^{2, \infty}(0, \ell)$
with $g > 0$ such that \eqref{eqn:ODE1} and \eqref{eqn:ODE2} hold true almost everywhere.
\end{enumerate}
\end{theorem}

To summarise, an $\infty$-elastica is
either a concatenation of two-di\-men\-sion\-al curves or a single three-dimensional curve solving a certain system of
differential equations.
In the first case, we have additional conditions that determine the curves to a
significant degree. For example, in the
case $\alpha \equiv 1$, it is readily seen that
any planar $\infty$-elastica comprises either
\begin{enumerate}[(a)]
\item \label{item:arc-line-arc}
a circular arc, followed by several line segments and full circles of equal radius, followed by a circular arc
(cf.\ Figure \ref{fig:clc}), or
\item \label{item:arc-arc} several circular arcs of equal length (apart from the first and the last) and radius but alternating
sense of rotation (cf.\ Figure \ref{fig:ccc1} and \ref{fig:ccc2}).
\end{enumerate}

\begin{figure}[h!tb]
\centering

\begin{subfigure}[t]{0.3\textwidth}
\includegraphics[width=\textwidth]{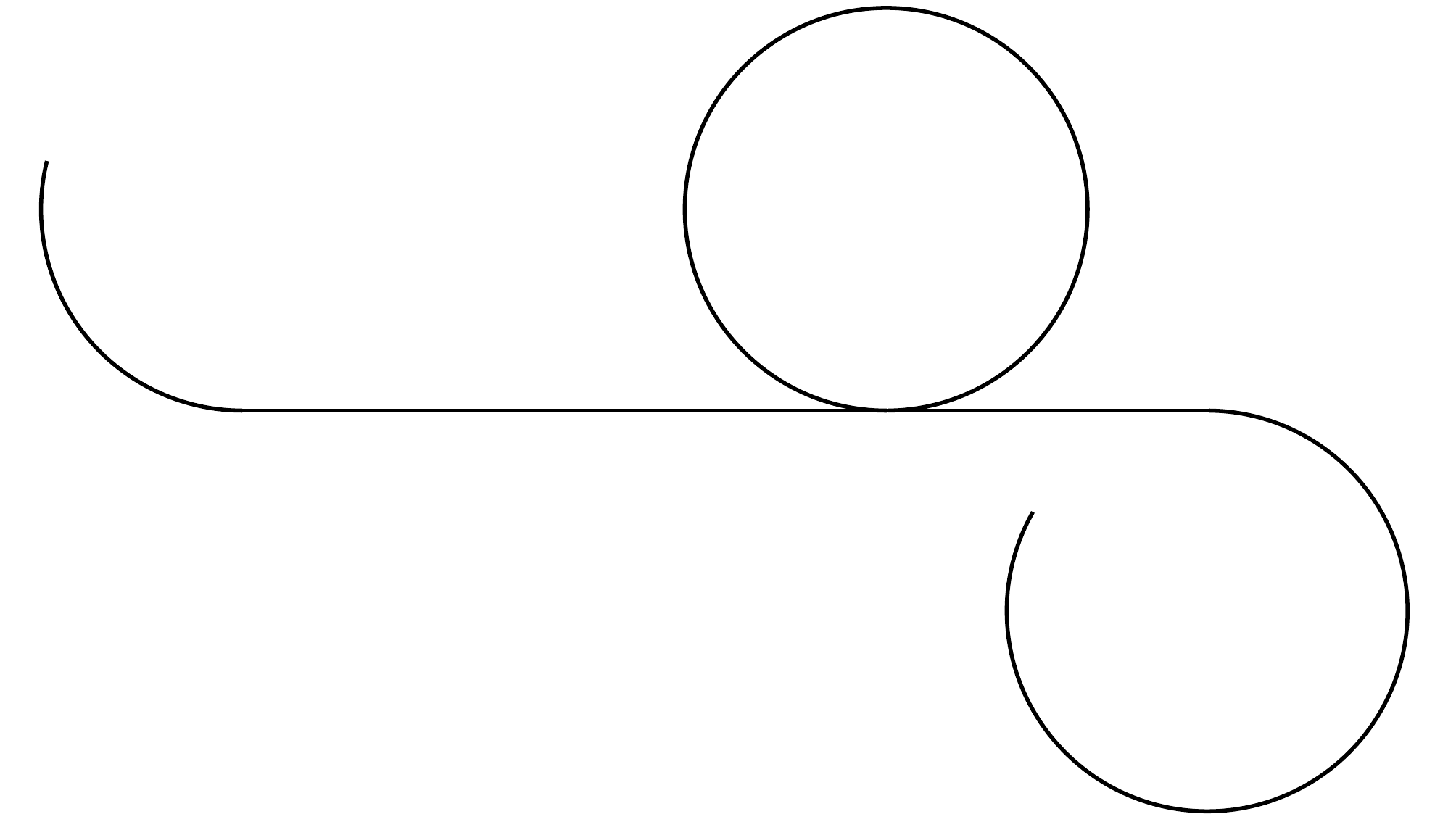}
\caption{$\lambda = (-1, 0)$}
\label{fig:clc}
\end{subfigure}
\hfill
\begin{subfigure}[t]{0.3\textwidth}
\includegraphics[width=\textwidth]{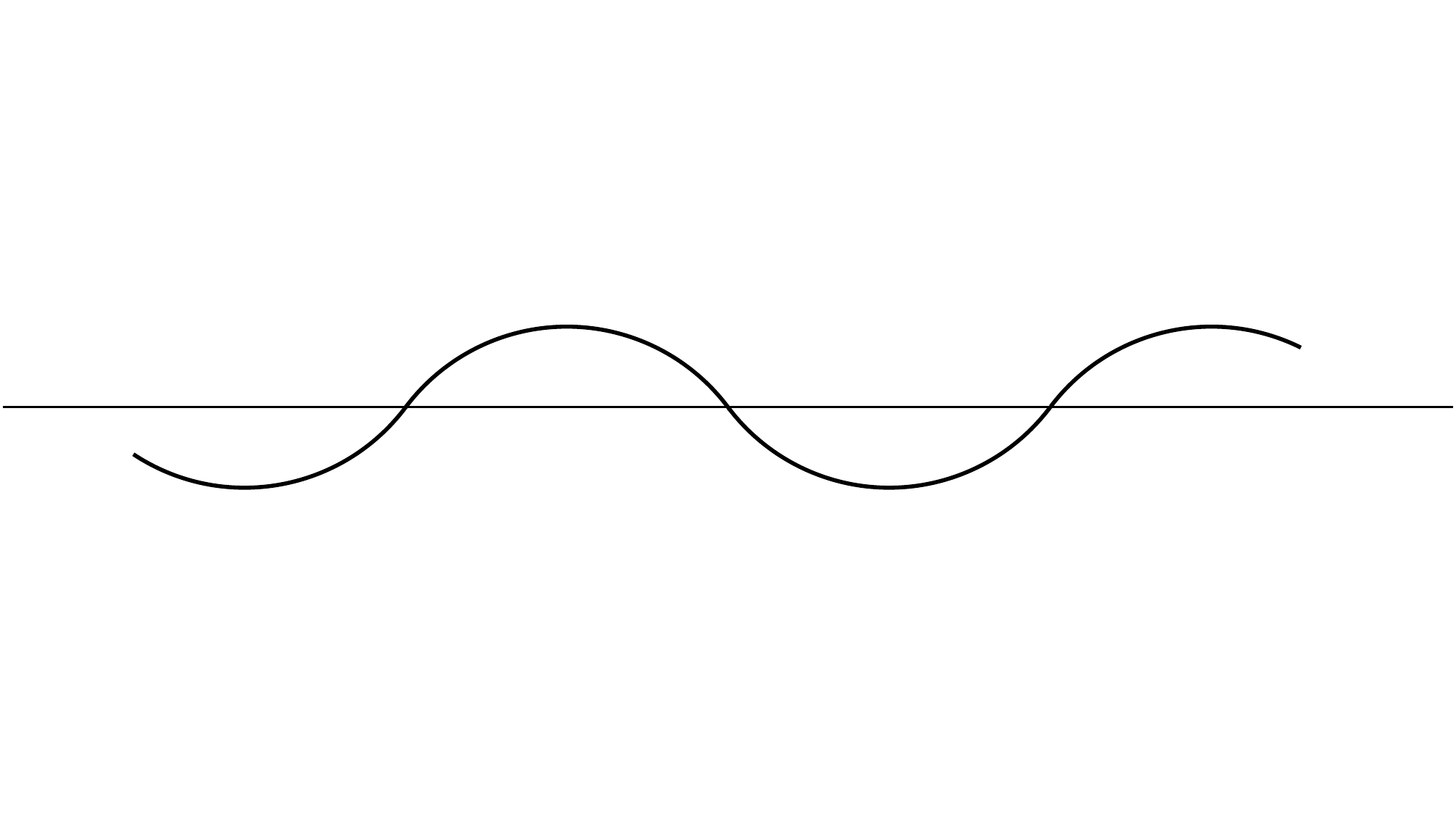}
\caption{$\lambda = (1, 0)$}
\label{fig:ccc1}
\end{subfigure}
\hfill
\begin{subfigure}[t]{0.3\textwidth}
\includegraphics[width=\textwidth]{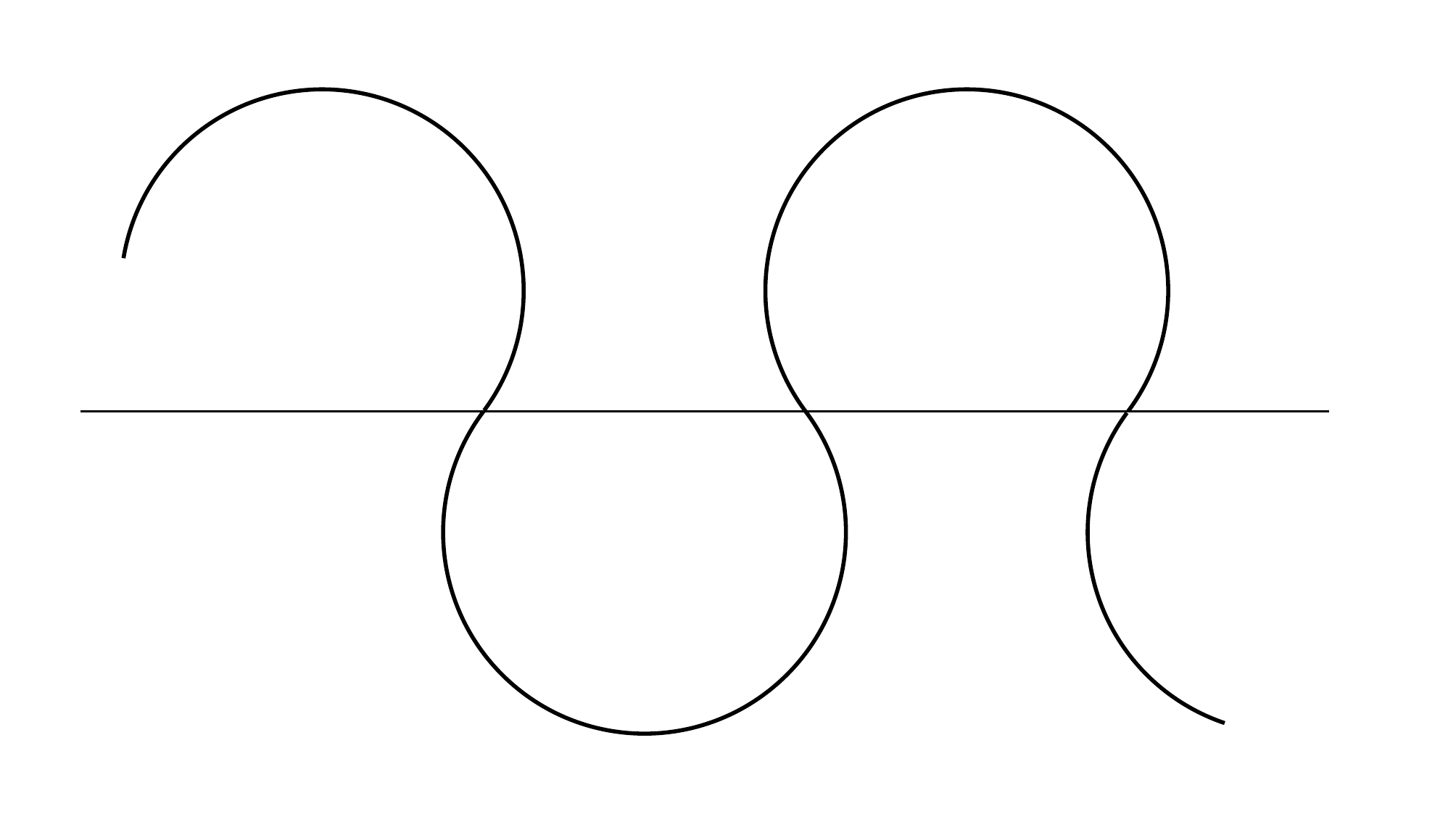}
\caption{$\lambda = (1, 0)$}
\label{fig:ccc2}
\end{subfigure}
\caption{These curves satisfy statement \ref{item:2D} of Theorem~\ref{thm:structure} for the $\lambda$
indicated. The parametrisation is from left to right in all three cases.}
\label{fig:circles&lines}

\end{figure}

Curves of both types, with the additional restriction that
they consist of at most three pieces, have been found by Dubins \cite{Dubins:57} as the solutions
of a different variational problem:
Dubins minimises the \emph{length} of a planar curve subject to boundary conditions of the
type \eqref{eqn:boundary_conditions} and subject to the constraint that the curvature
should nowhere exceed a given number.
This problem was previously considered by Markov \cite{Markov:1887}
and is therefore known as the Markov-Dubins problem.
Dubins calls the solutions \emph{$R$-geodesics}
if $1/R$ is the maximum curvature permitted. A similar result has been proved by
Sussmann \cite{Sussmann:95} in dimension $n = 3$. Just as in Theorem
\ref{thm:structure}, Sussmann finds two types of solutions: concatenations of
circles and line segments on the one hand and three-dimensional curves, that he calls helicoidal
arcs, on the other hand. The latter correspond to solutions of equation
\eqref{eqn:ODE_alpha=1}. Sussmann's proof relies on a reformulation of the problem as an
optimal control problem and on Pontryagin's maximum principle. For the problem studied in
this paper, such an approach seems to be unavailable.

It is no surprise that we obtain similar solutions, for the two problems are connected.

\begin{proposition}[$R$-geodesics minimise $\K_1$] \label{prop:shortest_curves}
Let $R > 0$. Suppose that $\gamma \colon [0, \ell] \to \R^n$ parametrises
an $R$-geodesic by arc length.
Then $\gamma$ minimises $\K_1$ subject to its boundary data.
\end{proposition}

As a consequence, we obtain an alternative proof of Dubins's and Sussmann's main results.
Theorem \ref{thm:structure} will initially give less information in case \ref{item:2D}, but the proofs
can then be completed with elementary arguments and some of Dubins's lemmas.
We give a sketch of these arguments in Section \ref{sect:Dubins}.

The Markov-Dubins problem, and variants thereof \cite{Reeds-Shepp:90}, have found applications in motion planning \cite{Laumond-Sekhavat-Lamiraux:98}.

There is a connection to another classical problem.
In 1925, Schmidt \cite{Schmidt:25} studied open spacial curves of fixed length
that minimise the length of the chord under the constraint that the curvature is
bounded pointwise by a given function (that we identify with $1/\alpha$). He generalised a result
of A. Schur \cite{Schur:21}, which in turn refines an unpublished result ascribed by both authors
to Schwarz. Another proof of this result may be found in a book of Blaschke \cite[\S 31]{Blaschke:45},
and a proof in English is given by S. S. Chern \cite{Chern:67}. The solutions
of this problem are obviously minimisers of $\K_\alpha$, too, even under weaker
boundary conditions.
Schmidt concludes that any curve with shortest chord subject to his curvature constraint
must be planar and convex. This can of course not be expected for
the variational problem with boundary conditions \eqref{eqn:boundary_conditions} in general.

The strategy for the proofs of Theorem \ref{thm:DE}--\ref{thm:structure} is to first approximate
the $L^\infty$-norm of the curvature by $L^p$-norms for $p < \infty$ and then let
$p \to \infty$. For $p < \infty$, we obtain a similar variational problem, which
gives rise to an Euler-Lagrange equation. When we pass to the limit $p \to \infty$,
the Euler-Lagrange equation is preserved in some form and eventually gives rise to
the system \eqref{eqn:ODE1}, \eqref{eqn:ODE2}. We also obtain some information
about the structure of solutions from the limit. A detailed analysis of the differential
equations is also necessary for Theorem \ref{thm:structure}.

To my knowledge, this is the first study of the above variational problem in the
literature, although, as already discussed, several related problems have been studied
in significant detail. There is also extensive work on variational problems involving
an $L^\infty$-norm in general, going back to the work of
Aronsson \cite{Aronsson:65, Aronsson:66, Aronsson:67}. An introduction with many further
references is given in a book by Katzourakis \cite{Katzourakis:15}.
Higher order problems have been studied more recently as well
\cite{Aronsson:10, Moser-Schwetlick:12, Sakellaris:17, Katzourakis-Pryer:18, Katzourakis-Pryer:19, Katzourakis-Moser:19, Katzourakis-Parini:17}, but there is a much smaller body of literature.
An approximation by $L^p$-norms, as in this paper, is common for variational problems in $L^\infty$,
but subsequently, most of the literature relies on methods and ideas quite different
from what is used here. Nevertheless, our approach has previously been deployed, too
\cite{Moser-Schwetlick:12, Sakellaris:17, Katzourakis-Moser:19, Katzourakis-Parini:17}. For comparison,
the paper of
Katzourakis and the author \cite{Katzourakis-Moser:19} studies functions
$u \colon \Omega \to \R$, for some domain $\Omega \subseteq \R^n$, that minimise
$\esssup_{x \in \Omega} |F(x, \Delta u(x))|$ for a given function $F$ under prescribed boundary data.
The paper describes the structure of minimisers, derives a system of partial differential
equations that characterises them, and proves that minimisers are unique.

For the problem studied here, it cannot be expected that minimisers are unique in
general, and this is one of the reasons why the previous methods are insufficient.
For example, if the boundary data are symmetric with respect to a reflection
(for $n = 2$) or rotation about a line (for $n > 2$), but $\ell$ is too long to admit
a straight line segment, then the symmetry of the problem automatically gives rise
to multiple solutions. Therefore, if we use approximations to the variational problem,
we will typically recover some solution in the limit, but not necessarily all possible solutions.
We overcome this difficulty by adding another term that penalises the distance
from a \emph{given} solution. This is the main novelty in the first part of our analysis. The
penalisation corresponds to the last term in the inequality of Definition \ref{def:infty-elastica},
and thus, although initially introduced as a technical device, proves to be interesting in its
own right, as it gives rise to a variational problem \emph{equivalent} to the system of
differential equations in Theorem~\ref{thm:DE}.

The second part of our analysis, which leads to the proof of Theorem \ref{thm:structure},
is completely new. The underlying method may be restricted to this and similar problems, but
our theory provides one of the first examples (the equally restrictive and more elementary theory of
Katzourakis-Pryer \cite[Section 8]{Katzourakis-Pryer:18} being the only other example I am aware of),
where a non-trivial second-order variational problem in $L^\infty$ can be
solved exhaustively.

\section{Reparametrisation and approximation} \label{sect:reparametrisation}

In this section, we prepare the ground for the proofs of Theorems \ref{thm:DE}--\ref{thm:structure}.
We first reformulate the problem by reparametrising the curves appropriately.
Then we discuss an approximation of the $L^\infty$-norm by $L^p$-norms.
We also add a penalisation term to the functionals, the purpose of which is to
guarantee convergence to a \emph{given} (rather than an arbitrary) solution of
the problem as $p \to \infty$. At the same time, we shift our main attention from a curve in $\R^n$ to
its tangent vector field.

Recall that we previously considered parametrisations $\gamma \colon [0, \ell] \to \R^n$ by
arc length satisfying the boundary conditions \eqref{eqn:boundary_conditions}.
From now on, a parametrisation with speed $\alpha$ is more convenient.
Therefore, define
\[
\psi(s) = \int_0^s \frac{d\sigma}{\alpha(\sigma)}, \quad 0 \le s \le \ell,
\]
and $L = \psi(\ell)$. Also consider the inverse $\phi = \psi^{-1} \colon [0, L] \to [0, \ell]$
and $\beta = \alpha \circ \phi$.
If $\gamma$ is a parametrisation by arc length, then the reparametrisation $c \colon [0, L] \to \R^n$,
given by $c(t) = \gamma(\phi(t))$, satisfies $|c'(t)| = \phi'(t) = 1/\psi'(\phi(t)) = \beta(t)$.

We now consider the tangent vector field along $c$, normalised to unit length. Thus
let $\tau \colon [0, L] \to S^{n - 1}$ be defined by $\tau(t) = c'(t)/\beta(t)$.
(An equivalent definition is $\tau(t) = \gamma'(\phi(t))$.)
Then \eqref{eqn:boundary_conditions} implies that
\begin{equation} \label{eqn:boundary_condition1}
\tau(0) = T_1 \quad \text{and} \quad \tau(L) = T_2.
\end{equation}
Setting $a = a_2 - a_1$, we also obtain the condition
\begin{equation} \label{eqn:boundary_condition2}
\int_0^L \beta(t) \tau(t) \, dt = a.
\end{equation}

Conversely, if we have $\tau \in W^{1, \infty}((0, L); S^{n - 1})$ satisfying
\eqref{eqn:boundary_condition1} and \eqref{eqn:boundary_condition2}, then
$\gamma \in \mathcal{G}$ can be reconstructed from $\tau$ by
\[
\gamma(s) = a_1 + \int_0^s \tau(\psi(\sigma)) \, d\sigma, \quad 0 \le s \le \ell.
\]
The functional $\K_\alpha$ can be written in terms of $\tau$ as follows:
\[
\K_\alpha(\gamma) = \esssup |\tau'|.
\]
Hence in order to study the above problem, it suffices to consider $\tau$ and to study the functional
\[
K_\infty(\tau) = \esssup |\tau'|
\]
under the boundary conditions \eqref{eqn:boundary_condition1} and the integral
constraint \eqref{eqn:boundary_condition2}. We note that $\gamma$ is an $\infty$-elastica
if, and only if, $\tau$ has the following property.

\begin{definition}
Suppose that $\tau \in W^{1, \infty}((0, L); S^{n - 1})$ satisfies the boundary conditions
\eqref{eqn:boundary_condition1} and the constraint \eqref{eqn:boundary_condition2}.
We say that $\tau$ is a \emph{pseudo-minimiser}
of $K_\infty$ if there exists $m \in \R$ such that
\[
K_\infty(\tau) \le K_\infty(\tilde{\tau}) + \frac{m}{2L} \int_0^L \beta |\tilde{\tau} - \tau|^2 \, dt
\]
for any other $\tilde{\tau} \in W^{1, \infty}((0, L); S^{n - 1})$ satisfying
\eqref{eqn:boundary_condition1} and \eqref{eqn:boundary_condition2}.
\end{definition}

One of the key tools for the proofs of Theorems \ref{thm:DE}--\ref{thm:structure} is an approximation of $K_\infty$ by
\[
K_p(\tau) = \left(\frac{1}{L} \int_0^L |\tau'|^p \, dt\right)^{1/p}
\]
for $p \in [2, \infty)$. We eventually consider the limit as $p \to \infty$ to recover $K_\infty$.
Furthermore, given $\tau_0 \in W^{1, \infty}((0, L); S^{n - 1})$ and $\mu \ge 0$, we consider the functionals
\[
J_p^\mu(\tau; \tau_0) = K_p(\tau) + \frac{\mu}{2L} \int_0^L \beta |\tau - \tau_0|^2 \, dt.
\]
In the proofs of Theorems \ref{thm:DE}--\ref{thm:structure}, we will assume that $\tau_0$ is a pseudo-minimiser of $K_\infty$.
Minimisers of $J_p^\mu(\blank; \tau_0)$ can then be found with the direct method, and the assumption
will guarantee that they converge to $\tau_0$ as $p \to \infty$. This will eventually allow some conclusions about $\tau_0$.
Indeed, the following preliminary observations are almost immediate from the structure of the variational problem.

\begin{proposition} \label{prop:p-approximation}
Let $\mu > 0$ and $\tau_0 \in W^{1, \infty}((0, L); S^{n - 1})$ be given. For every $p \in [2, \infty)$,
suppose that $\tau_p \in W^{1, p}((0, L); S^{n - 1})$ is a minimiser of $J_p^\mu(\blank; \tau_0)$
subject to the constraints \eqref{eqn:boundary_condition1} and \eqref{eqn:boundary_condition2}
and let $k_p = K_p(\tau_p)$.
\begin{enumerate}
\item \label{item:Euler-Lagrange} Then there are Lagrange multipliers
$\Lambda_p \in \R^n$ such that
\begin{equation} \label{eqn:Euler-Lagrange}
\frac{d}{dt} \left(|\tau_p'|^{p - 2} \tau_p'\right) + |\tau_p'|^p \tau_p = k_p^{p - 1} \beta \bigl(\Lambda_p - (\Lambda_p \cdot \tau_p) \tau_p - \mu \tau_0 + \mu (\tau_0 \cdot \tau_p)\tau_p\bigr)
\end{equation}
weakly in $(0, L)$.
\item \label{item:convergence}
If $\tau_0$ satisfies \eqref{eqn:boundary_condition1} and \eqref{eqn:boundary_condition2}
and is a pseudo-minimiser of $K_\infty$, then there exists $\mu_0 > 0$ such that the
following holds true. If $\mu \ge \mu_0$, then
$\tau_p \rightharpoonup \tau_0$ weakly in $W^{1, q}((0, L); \R^n)$ for every $q < \infty$
and $k_p \to K_\infty(\tau_0)$ as $p \to \infty$.
\end{enumerate}
\end{proposition}

\begin{proof}
The Euler-Lagrange equation \eqref{eqn:Euler-Lagrange} is derived with
standard computations. The only feature that is perhaps unusual is the constraint
$\tau_p(t) \in S^{n - 1}$ for $t \in [0, L]$, but this sort of constraint is common
in the theory of harmonic maps and it is explained, e.g., in a book by
Simon \cite{Simon:96} how to deal with it. We therefore omit the details in the
proof of statement \ref{item:Euler-Lagrange}.

Next we note that by the choice of $\tau_p$ and by H\"older's inequality,
for any pair of numbers $p, q \in (1, \infty)$ with
$p \le q$, we find the inequalities
\begin{equation} \label{eqn:monotonicity}
J_p^\mu(\tau_p; \tau_0) \le J_p^\mu(\tau_q; \tau_0)\le J_q^\mu(\tau_q; \tau_0) \le K_q(\tau_0) \le K_\infty(\tau_0).
\end{equation}
So for any $q \in [2, \infty)$, the one-parameter family $(\tau_p)_{q \le p < \infty}$ is bounded in $W^{1, q}((0, L); \R^n)$.
Therefore, there exists a sequence $p_i \to \infty$ such that $\tau_{p_i}$ converges weakly in $W^{1, q}((0, L); \R^n)$,
for every $q < \infty$, to a limit
\[
\tau_\infty \in \bigcap_{q < \infty} W^{1, q}((0, L); S^{n - 1}).
\]
Clearly $\tau_\infty$ will satisfy \eqref{eqn:boundary_condition1} and \eqref{eqn:boundary_condition2} again.
By the lower semicontinuity of the $L^q$-norm with respect to weak convergence and by \eqref{eqn:monotonicity},
\begin{equation} \label{eqn:upper_bound}
J_\infty^\mu(\tau_\infty; \tau_0) = \lim_{q \to \infty} J_q^\mu(\tau_\infty; \tau_0) \le \lim_{q \to \infty} \liminf_{i \to \infty} J_q^\mu(\tau_{p_i}; \tau_0) \le K_\infty(\tau_0).
\end{equation}

If there exists $m > 0$ such that
\[
K_\infty(\tau_0) \le K_\infty(\tau) + \frac{m}{2L} \int_0^L \beta |\tau - \tau_0|^2 \, dt
\]
for all $\tau \in W^{1, \infty}((0, L); S^{n - 1})$ satisfying \eqref{eqn:boundary_condition1} and \eqref{eqn:boundary_condition2},
then \eqref{eqn:upper_bound} implies that
\[
J_\infty^\mu(\tau_\infty; \tau_0) \le J_\infty^m(\tau_\infty; \tau_0).
\]
Thus
\[
(\mu - m) \int_0^L \beta |\tau_\infty - \tau_0|^2 \, dt \le 0.
\]
If we choose $\mu > m$, this means that $\tau_\infty = \tau_0$. In particular, the limit
is then independent of the choice of the sequence $(p_i)_{i \in \N}$, and therefore we have in fact weak convergence
of $\tau_p$ to $\tau_\infty = \tau_0$ in $W^{1, q}((0, L); \R^n)$ for every $q < \infty$.

The inequalities in \eqref{eqn:monotonicity} also imply that
\[
\lim_{p \to \infty} J_p^\mu(\tau_p; \tau_0) \le K_\infty(\tau_0),
\]
in particular that the limit exists. On the other hand,
as we now know that $\tau_\infty = \tau_0$, we can go back to \eqref{eqn:upper_bound} and conclude that
\[
K_\infty(\tau_0) \le \lim_{q \to \infty} \liminf_{i \to \infty} J_q^\mu(\tau_{p_i}; \tau_0) \le \lim_{p \to \infty} J_p^\mu(\tau_p; \tau_0).
\]
Hence $K_\infty(\tau_0) = \lim_{p \to \infty} J_p^\mu(\tau_p; \tau_0)$. Since
the weak convergence $\tau_p \rightharpoonup \tau_0$ in $W^{1, 2}((0, L); \R^n)$
implies strong convergence in $L^2((0, L); \R^n)$ as well, it follows that
$K_\infty(\tau_0) = \lim_{p \to \infty} k_p$.
\end{proof}

Eventually we will need a careful analysis of the Euler-Lagrange equation
\eqref{eqn:Euler-Lagrange} for the proofs of Theorems \ref{thm:DE}--\ref{thm:structure}. To this
end, we need to know that the Lagrange multipliers $\Lambda_p$ do not grow
too quickly as $p \to \infty$. We prove the following.

\begin{lemma} \label{lem:growth}
Suppose that $\tau_p \in W^{1, p}((0, L); S^{n - 1})$ and let $k_p = K_p(\tau_p)$. Suppose that
$\limsup_{p \to \infty} k_p < \infty$ and there exist $\Lambda_p \in \R^n$ such that
\eqref{eqn:Euler-Lagrange} holds weakly in $(0, L)$ for every $p \in [2, \infty)$.
Then either
\[
\limsup_{p \to \infty} \left(p^{-6} |\Lambda_p|\right) < \infty
\]
or there exists a sequence $p_i \to \infty$ such that
$\tau_{p_i}$ converges uniformly to a constant vector as $i \to \infty$.
\end{lemma}

\begin{proof}
Suppose that no subsequence converges uniformly to a constant vector.
Then it follows that for every sufficiently large $p$, either $\Lambda_p = 0$
or the angle $\omega_p$ between $\tau_p$ and $\Lambda_p$ satisfies
\[
\sup_{t \in [0, L]} \omega_p(t) \ge \frac{1}{p} \quad \text{and} \quad \sup_{t \in [0, L]} (\pi - \omega_p(t)) \ge \frac{1}{p}.
\]
Note that
\[
|\sin \omega_p| = \frac{|\Lambda_p - (\Lambda_p \cdot \tau_p) \tau_p|}{|\Lambda_p|}
\]
if $\Lambda_p \neq 0$.
Hence for every sufficiently large $p$, there exists $t_p \in [0, L]$ such that
\[
\left|\Lambda_p - (\Lambda_p \cdot \tau_p(t_p)) \tau_p(t_p)\right| \ge \frac{|\Lambda_p|}{2p}.
\]

Because we have a uniform bound for $\|\tau_p'\|_{L^2(0, L)}$, the Sobolev
embedding theorem gives a uniform bound for $\|\tau_p\|_{C^{0, 1/2}([0, T])}$ as well.
Hence there exists
a number $\delta > 0$ such that the inequality
\[
\left|\Lambda_p - (\Lambda_p \cdot \tau_p) \tau_p\right| \ge \frac{|\Lambda_p|}{3p}
\]
holds in $[t_p - \delta/p^2, t_p + \delta/p^2] \cap [0, L]$ for all sufficiently large values of $p$.
Choose $\eta \in C_0^\infty((t_p - \delta/p^2, t_p + \delta/p^2) \cap (0, L))$ such that $0 \le \eta \le 1$ and
\[
\int_0^L \eta \, ds \ge \frac{\delta}{2p^2},
\]
but $|\eta'| \le 5p^2/\delta$. Test \eqref{eqn:Euler-Lagrange} with $\eta \Lambda_p$.
This yields
\begin{multline*}
\int_0^L \eta |\tau_p'|^p \tau_p \cdot \Lambda_p \, dt - \int_0^L \eta' |\tau_p'|^{p - 2} \tau_p' \cdot \Lambda_p \, dt \\
= k_p^{p - 1} \int_0^L \eta \beta \left|\Lambda_p - (\Lambda_p \cdot \tau_p) \tau_p\right|^2 \, dt - \mu k_p^{p - 1} \int_0^L \eta \beta \left(\tau_0 - (\tau_0 \cdot \tau_p) \tau_p\right) \cdot \Lambda_p \, dt.
\end{multline*}
By the choice of $\eta$, we know that
\[
\int_0^L \eta \beta \left|\Lambda_p - (\Lambda_p \cdot \tau_p) \tau_p\right|^2 \, dt \ge \frac{\delta |\Lambda_p|^2}{18p^4 \|1/\alpha\|_{L^\infty(0, \ell)}}.
\]
Moreover, we have the estimates
\begin{align*}
\int_0^L \eta |\tau_p'|^p \tau_p \cdot \Lambda_p \, dt & \le Lk_p^p |\Lambda_p|, \\
- \int_0^L \eta' |\tau_p'|^{p - 2} \tau_p' \cdot \Lambda_p \, dt & \le \frac{5p^2}{\delta} L k_p^{p - 1} |\Lambda_p|, \\
\int_0^L \eta \beta \left(\tau_0 - (\tau_0 \cdot \tau_p) \tau_p\right) \cdot \Lambda_p \, dt & \le L \|\alpha\|_{L^\infty(0, \ell)} |\Lambda_p|.
\end{align*}
Hence
\[
|\Lambda_p| \le \frac{18Lp^4}{\delta} \|1/\alpha\|_{L^\infty(0, \ell)} \left(\frac{5p^2}{\delta} + \mu \|\alpha\|_{L^\infty(0, \ell)} + k_p\right),
\]
and the desired inequality follows.
\end{proof}

\section{Preliminary properties of $\infty$-elasticas}

The purpose of this section is to extract some information for pseudo-minimisers of $K_\infty$,
and therefore for $\infty$-elasticas, from the Euler-Lagrange equation \eqref{eqn:Euler-Lagrange}
by studying the limit $p \to \infty$. The resulting statements are less strong than the main
results in the introduction, but they will serve as a first step.

\begin{proposition} \label{prop:pseudo-minimiser=>DE}
Suppose that $\tau \in W^{1, \infty}((0, L); S^{n - 1})$ is a pseudo-minimiser of $K_\infty$.
Let $k = K_\infty(\tau)$.
Then there exist $u \in W^{1, \infty}((0, L); \R^n) \setminus \{0\}$ and $\lambda \in \R^n$ such that
the equations
\begin{align}
u' + (u \cdot \tau') \tau & = \beta(\lambda - (\lambda \cdot \tau) \tau), \label{eqn:system1}\\
|u|\tau' & = k u, \label{eqn:system2}
\end{align}
hold almost everywhere in $(0, L)$.
\end{proposition}

\begin{proof}
The statements are obvious (for $u = 1$ and $\lambda = 0$) if $\tau$ is constant.
We therefore assume that this is not the case.

Fix $\mu > 0$ and consider the functionals
$J_p^\mu(\blank; \tau)$. Minimisers $\tau_p$ of $J_p^\mu(\blank; \tau)$ under the boundary conditions
\eqref{eqn:boundary_condition1} and the constraint \eqref{eqn:boundary_condition2} can be constructed
with the direct method. Let $k_p = K_p(\tau_p)$. We assume that $\mu > 0$ is chosen so large that statement \ref{item:convergence}
in Proposition \ref{prop:p-approximation} applies.

We consider the Euler-Lagrange equation \eqref{eqn:Euler-Lagrange}. The underlying
idea for the next step is to regard it as an equation in $|\tau_p'|^{p - 2} \tau_p'$.
But at the same time, we renormalise. Thus we introduce the functions
\[
u_p = \frac{k_p^{1 - p} |\tau_p'|^{p - 2} \tau_p'}{1 + |\Lambda_p|}.
\]
We also define
\[
\lambda_p = \frac{\Lambda_p}{1 + |\Lambda_p|} \quad \text{and} \quad m_p = \frac{\mu}{1 + |\Lambda_p|}.
\]
Then we can write \eqref{eqn:Euler-Lagrange} (for $\tau_0 = \tau$) in the form
\begin{equation} \label{eqn:Euler-Lagrange2}
u_p' + (u_p \cdot \tau_p') \tau_p = \beta(\lambda_p - (\lambda_p \cdot \tau_p) \tau_p - m_p \tau + m_p (\tau \cdot \tau_p)\tau_p).
\end{equation}

Writing $p' = p/(p - 1)$, we note that
\begin{equation} \label{eqn:L^p'-estimate}
\|u_p\|_{L^{p'}(0, L)} = \frac{k_p^{p - 1}}{1 + |\Lambda_p|} \left(\int_0^L |\tau_p'|^p \, dt\right)^{1/p'} = \frac{L^{1/p'}}{1 + |\Lambda_p|}.
\end{equation}
The right-hand side remains bounded as $p \to \infty$. Moreover, we know that
\[
\|\tau_p'\|_{L^p(0, L)} = L^{1/p} k_p \to k
\]
as $p \to \infty$ by Proposition \ref{prop:p-approximation}.
As $|\tau_p| \equiv 1$, $|\lambda_p| \le 1$, and $0 < m_p \le \mu$,
equation \eqref{eqn:Euler-Lagrange2} immediately gives a uniform bound for $\|u_p\|_{W^{1, 1}(0, L)}$.
Thus we have a uniform bound in $L^\infty((0, L); \R^n)$ as well, and using the equation again, we conclude
that
\[
\limsup_{p \to \infty} \|u_p'\|_{L^q(0, L)} < \infty
\]
for any $q < \infty$.
Thus we may choose a sequence $p_i \to \infty$ such that $u_{p_i} \rightharpoonup u$, for some
$u \in \bigcap_{q < \infty} W^{1, q}((0, L); \R^n)$,
weakly in $W^{1, q}((0, L); \R^n)$ for any $q < \infty$ as $i \to \infty$. In particular
$u_{p_i} \to u$ uniformly as $i \to \infty$. Since $|\lambda_p| \le 1$ and $0 < m_p \le \mu$, we may assume that at the same time, we have the convergence
$\lambda_{p_i} \to \lambda$ for some $\lambda \in \R^n$ and $m_{p_i} \to m$
for some $m \in [0, \mu]$.
By Proposition \ref{prop:p-approximation}, we know that $\tau_p \to \tau$ weakly in $W^{1, q}((0, L); \R^n)$
for any $q < \infty$. Thus restricting \eqref{eqn:Euler-Lagrange2} to $p_i$ and letting
$i \to \infty$, we derive
equation \eqref{eqn:system1} almost everywhere. Now \eqref{eqn:system1} implies that $u \in W^{1, \infty}((0, L); \R^n)$.

If $|\Lambda_{p_i}| \to \infty$ as $i \to \infty$, then $\lambda \in S^{n - 1}$ and \eqref{eqn:system1}
cannot be satisfied for $u \equiv 0$ (as we have assumed that $\tau$ is not constant).
If $|\Lambda_{p_i}| \not\to \infty$, then \eqref{eqn:L^p'-estimate} implies that $\|u\|_{L^1(0, L)} \neq 0$.
In either case, we conclude that $u \in W^{1, \infty}((0, L); \R^n) \setminus \{0\}$.

As $u$ is continuous, the set $\Omega = \set{t \in [0, L]}{u(t) \neq 0}$ is open
relative to $[0, L]$. For any
$t \in \Omega$, there exist $\delta > 0$ and $\epsilon > 0$ such that $\delta \le |u_{p_i}| \le 1/\delta$
in $(t - \epsilon, t + \epsilon) \cap [0, L]$ for any $i$ large enough. Now note that
\[
\tau_p' = k_p (1 + |\Lambda_p|)^{1/(p - 1)} |u_p|^{1/(p - 1)} \frac{u_p}{|u_p|}
\]
wherever $u_p \neq 0$ by the definition of $u_p$. As we have assumed that $\tau$ is not constant, we know that
\[
(1 + |\Lambda_p|)^{1/(p - 1)} \to 1
\]
as $p \to \infty$ by Lemma \ref{lem:growth}.
We further know that
\[
|u_{p_i}|^{1/(p_i - 1)} \to 1 \quad \text{and} \quad \frac{u_{p_i}}{|u_{p_i}|} \to \frac{u}{|u|}
\]
uniformly in $(t - \epsilon, t + \epsilon) \cap [0, L]$ as $i \to \infty$. Therefore, by the above identity,
\[
\tau_{p_i}' \to \tau' = \frac{k u}{|u|}
\]
locally uniformly in $\Omega$. We therefore obtain equation \eqref{eqn:system2}.
\end{proof}

For planar curves, we can say more.

\begin{lemma} \label{lem:2D}
Let $\tau \in W^{1, \infty}((0, L); S^{n - 1})$ and $\lambda \in \R^n \setminus \{0\}$. Suppose that
$\tau([0, L])$ is contained in a two-dimensional linear subspace $X \subseteq \R^n$.
Let
\[
c(t) = a_1 + \int_0^t \beta(\theta) \tau(\theta) \, d\theta
\]
for $t \in [0, L]$. Suppose that $k = K_\infty(\tau) \neq 0$ and consider a set $\Omega \subseteq [0, L]$.
Then the following statements are equivalent.
\begin{enumerate}[(i)]
\item \label{item:system2D}
There exists $u \in W^{1, \infty}((0, L); \R^n)$ such that \eqref{eqn:system1} and \eqref{eqn:system2}
hold true almost everywhere and $\Omega = \set{t \in [0, L]}{u(t) \neq 0}$.

\item \label{item:structure2D}
The vector $\lambda$ belongs to $X$ and there exists a line $\mathcal{L} \subseteq X + a_1$ parallel to $\lambda$ such
that $\Omega = \set{t \in [0, L]}{c(t) \not\in \mathcal{L}}$. Moreover,
$\tau'$ is continuous with $|\tau'| \equiv k$ in $\Omega$. For any $t_0 \in [0, L] \setminus \Omega$,
if there exists $\delta > 0$
with $(t_0 - \delta, t_0) \subseteq \Omega$, then there exists $\delta' \in (0, \delta]$ such that
$\lambda \cdot \tau'(t) < 0$ in $(t_0 - \delta', t_0)$;
and if there exists $\delta > 0$
with $(t_0, t_0 + \delta) \subseteq \Omega$, then there exists $\delta' \in (0, \delta]$ such that
$\lambda \cdot \tau'(t) > 0$ in $(t_0, t_0 + \delta')$.
\end{enumerate}
\end{lemma}

\begin{proof}
We may choose coordinates such that $X = \R^2 \times \{0\}$ and then write
\[
\tau = (\cos \omega, \sin \omega, 0)
\]
in $[0, L]$ for some function $\omega \colon [0, L] \to \R$.
Now for $x \in \R^n$, write $x^\perp = (-x_2, x_1, x_3, \dotsc, x_n)$. In particular
$\tau^\perp = (-\sin \omega, \cos \omega, 0)$ and $\tau' = \omega' \tau^\perp$.

If \ref{item:system2D} is satisfied, then \eqref{eqn:system2} implies that $u(t) \in X$ for every
$t \in [0, L]$, and then \eqref{eqn:system1} implies that $\lambda \in X$. It is clear that $u/|u|$ is continuous in $\Omega$.
Thus equation \eqref{eqn:system2} further implies that $\omega'$ is continuous in $\Omega$ with $|\omega'| \equiv k$.
Defining $f = |u| \omega'/k$, we compute $u = f\tau^\perp$ and
\begin{equation} \label{eqn:system1_2D}
u' + (u \cdot \tau') \tau = f'\tau^\perp.
\end{equation}
Multiplying \eqref{eqn:system1} with $\tau^\perp$, we conclude that
\begin{equation} \label{eqn:f2D}
f' = \beta \lambda \cdot \tau^\perp
\end{equation}
in $\Omega$. Outside of $\Omega$, we know that $f$ vanishes, and it follows that for any $t_1, t_2 \in [0, L]$,
we have the inequality $|f(t_1) - f(t_2)| \le \|\beta\|_{L^\infty(0, L)} |\lambda| |t_1 - t_2|$.
So $f \in W^{1, \infty}(0, L)$ and \eqref{eqn:system1_2D}, \eqref{eqn:f2D} hold true almost everywhere in $[0, L]$.

Consider $c$ as defined above and note that $(c')^\perp = \beta \tau^\perp$.
Hence $f' = \lambda \cdot (c')^\perp$ in $[0, L]$. It follows that there exists some number $b \in \R$ such that
\[
f^{-1}(\{0\}) = \set{t \in [0, L]}{\lambda^\perp \cdot c(t) = b}.
\]
In other words, the line $\mathcal{L} = \set{x \in X + a_1}{\lambda^\perp \cdot x = b}$,
which is parallel to $\lambda$, has the property that
$\Omega = \set{t \in [0, L]}{c(t) \not\in \mathcal{L}}$.

Now suppose that $t_0 \in [0, L] \setminus \Omega$ such that there exists $\delta > 0$ with $(t_0 - \delta, t_0) \subseteq \Omega$.
Recall that $|\omega'| \equiv k$ in $(t_0 - \delta, t_0)$ while the sign of $\omega'$ is constant. So
$\omega' = \sigma k$ in $(t_0 - \delta, t_0)$ for some $\sigma \in \{-1, 1\}$. Hence
\begin{equation} \label{eqn:tau2D}
\tau(t) = \cos(k(t - t_0)) \tau(t_0) + \sigma \sin(k(t - t_0)) \tau^\perp(t_0)
\end{equation}
and
\[
\tau'(t) = -k\sin(k(t - t_0)) \tau(t_0) + \sigma k\cos(k(t - t_0)) \tau^\perp(t_0)
\]
in $(t_0 - \delta, t_0)$. Moreover, identity \eqref{eqn:f2D} implies that
\begin{equation} \label{eqn:f'2D}
f'(t) = -\sigma \beta(t) \sin(k(t - t_0)) \lambda \cdot \tau(t_0) + \beta(t) \cos(k(t - t_0)) \lambda \cdot \tau^\perp(t_0)
\end{equation}
in $(t_0 - \delta, t_0)$. As $f(t_0) = 0$ and as $f$ has the same sign as $\omega'$ in $(t_0 - \delta, t_0)$,
we immediately conclude that $\sigma \lambda \cdot \tau^\perp(t_0) \le 0$; and in the case of equality,
we further conclude that $\lambda \cdot \tau(t_0) < 0$. But then, as
\[
\lambda \cdot \tau'(t) = -k\sin(k(t - t_0)) \lambda \cdot \tau(t_0) + \sigma k\cos(k(t - t_0)) \lambda \cdot \tau^\perp(t_0),
\]
this implies that $\lambda \cdot \tau'(t) < 0$ in $(t_0 - \delta', t_0)$ for some $\delta' > 0$.
If there exists $\delta > 0$ such that $(t_0, t_0 + \delta) \subseteq \Omega$, then we
can draw similar conclusions with the same arguments.
Hence \ref{item:structure2D} is satisfied.

Conversely, suppose that \ref{item:structure2D} holds true. If $c([0, L]) \subseteq \mathcal{L}$, set $u = 0$.
Otherwise, set
\[
f(t) = \lambda \cdot c^\perp(t) + b,
\]
where $b \in \R$ is chosen such that $\Omega = f^{-1}(\{0\})$. Then $f \in W^{1, \infty}(0, L)$
and \eqref{eqn:f2D} is satisfied.
If $(t_0, t_1) \subseteq \Omega$ is any connected component of $\Omega$, then $\omega' = \sigma k$
in $(t_0, t_1)$ for some fixed $\sigma \in \{-1, 1\}$. Hence we can write $\tau$ in the form
\eqref{eqn:tau2D} and it follows that $f'$ satisfies \eqref{eqn:f'2D}
in $(t_0, t_1)$.
The condition on the sign of $\lambda \cdot \tau'$ near $t_0$ implies that
$\sigma \lambda \cdot \tau^\perp(t_0) \ge 0$; and in the case of equality,
it also implies that $\lambda \cdot \tau(t_0) < 0$. Therefore, the function $f$ has the same sign as
$\omega'$ in $(t_0, t_1)$. Similar conclusions hold if we have connected components of
$\Omega$ of the form $[0, t_1)$ or $(t_0, L]$. Hence
$f$ and $\omega'$ have the same sign everywhere in $\Omega$.

Now we set $u = f\tau^\perp$. Then \eqref{eqn:system2} is obvious and
\eqref{eqn:system1} can be verified by computing \eqref{eqn:system1_2D} again and observing that
\[
\beta(\lambda - (\lambda \cdot \tau) \tau) = \beta (\lambda \cdot \tau^\perp) \tau^\perp = (\lambda \cdot (c')^\perp) \tau^\perp = f' \tau^\perp.
\]
This concludes the proof.
\end{proof}

\section{Analysis of the differential equations}

In this section we study the system \eqref{eqn:system1}, \eqref{eqn:system2} and its relationship
to the variational problem in more detail. Furthermore, we show that it is equivalent to \eqref{eqn:ODE1}, \eqref{eqn:ODE2}
up to the reparametrisation introduced in Section \ref{sect:reparametrisation}.

\begin{proposition} \label{prop:DE=>pseudo-minimiser}
Suppose that $\tau \in W^{1, \infty}((0, L); S^{n - 1})$ satisfies
\eqref{eqn:boundary_condition1} and \eqref{eqn:boundary_condition2}.
If there exist $u \in W^{1, \infty}((0, L); \R^n) \setminus \{0\}$ and $\lambda \in \R^n$ such that
\eqref{eqn:system1} and \eqref{eqn:system2} hold almost everywhere in $(0, L)$, then
$\tau$ is a pseudo-minimiser of $K_\infty$. If in addition
$k|u| + \beta \lambda \cdot \tau \le 0$ in $[0, L]$, then $\tau$ is a minimiser of $K_\infty$ under the constraints \eqref{eqn:boundary_condition1} and \eqref{eqn:boundary_condition2}.
\end{proposition}

\begin{proof}
Suppose that equations \eqref{eqn:system1} and \eqref{eqn:system2} hold true.
Let $\Sigma = u^{-1}(\{0\})$. We claim that $\tau' = 0$ almost everywhere on $\Sigma$.
Indeed, if $\lambda = 0$, then it follows from \eqref{eqn:system1} that
$\left|\frac{d}{dt} |u|\right| \le |u| \|\tau'\|_{L^\infty(0, L)}$.
As it is assumed that $u \not\equiv 0$, this inequality implies that $u \neq 0$ throughout $[0, L]$.
If $\lambda \neq 0$, then at almost every point $t \in \Sigma$, either $u'(t) \neq 0$
(so $t$ is an isolated point of $\Sigma$) or $\tau(t) = \pm \lambda/|\lambda|$. As $\tau \in W^{1, \infty}((0, L); S^{n - 1})$,
it has a derivative almost everywhere and we conclude that $\tau' = 0$ almost everywhere in $\Sigma$.

Now consider a competitor $\tilde{\tau} \colon [0, L] \to S^{n - 1}$
satisfying \eqref{eqn:boundary_condition1} and \eqref{eqn:boundary_condition2}.
Let $\sigma = \tilde{\tau} - \tau$
and note that
\[
1 = |\tau + \sigma|^2 = 1 + 2 \tau \cdot \sigma + |\sigma|^2
\]
in $[0, L]$. Hence
\[
\tau \cdot \sigma = - \frac{|\sigma|^2}{2}.
\]
Furthermore, the definition of $\sigma$ guarantees that $\sigma(0) = \sigma(L) = 0$ and
\[
\int_0^L \beta \sigma \, dt = 0.
\]
Observing that $u \cdot \tau' = k|u|$ because of \eqref{eqn:system2}, we now use \eqref{eqn:system1} to compute
\begin{equation} \label{eqn:integral}
\begin{split}
\int_0^L \sigma' \cdot u \, dt & = - \int_0^L \sigma \cdot u' \, dt \\
& = \int_0^L \left((k|u| + \beta \lambda \cdot \tau) \tau \cdot \sigma - \beta \lambda \cdot \sigma\right) \, dt \\
& = -\frac{1}{2} \int_0^L (k|u|/\beta + \lambda \cdot \tau) \beta |\sigma|^2 \, dt \\
& \ge -\frac{1}{2} \left(k\|u/\beta\|_{L^\infty(0, L)} + |\lambda|\right) \bigl\|\sqrt{\beta} \sigma\bigr\|_{L^2(0, L)}^2.
\end{split}
\end{equation}
Set
\[
M = \frac{k\|u/\beta\|_{L^\infty(0, L)} + |\lambda|}{2\|u\|_{L^1(0, L)}}.
\]
Then there exists a set $A \subseteq [0, L]$ of positive measure such that
$\sigma' \cdot u \ge -M \|\sqrt{\beta} \sigma\|_{L^2(0, L)}^2 |u|$ and $u \neq 0$ in $A$.
(Otherwise, we would conclude that
\[
\begin{split}
\int_0^L \sigma' \cdot u \, dt & = \int_{(0, L) \setminus \Sigma} \sigma' \cdot u \, dt \\
& < -M \bigl\|\sqrt{\beta} \sigma\bigr\|_{L^2(0, L)}^2 \int_{(0, L) \setminus \Sigma} |u| \, dt \\
& = -M \bigl\|\sqrt{\beta} \sigma\bigr\|_{L^2(0, L)}^2 \int_0^L |u| \, dt \\
& = -\frac{1}{2} \left(k\|u/\beta\|_{L^\infty(0, L)} + |\lambda|\right) \bigl\|\sqrt{\beta} \sigma\bigr\|_{L^2(0, L)}^2,
\end{split}
\]
in contradiction to \eqref{eqn:integral}.)
Hence
\[
\sigma' \cdot \tau' = \sigma' \cdot \frac{k u}{|u|} \ge - kM \bigl\|\sqrt{\beta} \sigma\bigr\|_{L^2(0, L)}^2
\]
almost everywhere in $A$. As $|\tau'| = k$ almost everywhere in $A$, it follows that
\[
\begin{split}
|\tilde{\tau}'| & = \sqrt{|\tau'|^2 + 2\tau' \cdot \sigma' + |\sigma'|^2} \\
& \ge \sqrt{k^2 - 2kM \bigl\|\sqrt{\beta} \sigma\bigr\|_{L^2(0, L)}^2} \\
& \ge k - 2M\bigl\|\sqrt{\beta} \sigma\bigr\|_{L^2(0, L)}^2
\end{split}
\]
almost everywhere in $A$ (unless the right-hand side is negative, in which case the intermediate
expression should be replaced by $0$). In particular,
\[
K_\infty(\tilde{\tau}) \ge K_\infty(\tau) - 2M \int_0^L \beta |\tilde{\tau} - \tau'|^2 \, dt.
\]
That is, we have shown that $\tau$ is a pseudo-minimiser.

Finally, if $k|u| + \beta \lambda \cdot \tau \le 0$, we can improve
\eqref{eqn:integral} and conclude that
\[
\int_0^L \sigma' \cdot u \, dt \ge 0.
\]
So there exists a set of positive measure $A \subseteq [0, L]$ where $u \neq 0$ and $\sigma' \cdot \tau' \ge 0$.
Thus $|\tilde{\tau}'|^2 \ge |\tau'|^2 = k^2$ almost everywhere in $A$, and
it follows immediately that $K_\infty(\tau) \le K_\infty(\tilde{\tau})$.
\end{proof}

Next we reformulate the system \eqref{eqn:system1}, \eqref{eqn:system2}.
We obtain the system \eqref{eqn:ODE_rescaled1}, \eqref{eqn:ODE_rescaled2} below,
which corresponds to \eqref{eqn:ODE1}, \eqref{eqn:ODE2} up to the reparametrisation
from Section~\ref{sect:reparametrisation}.

\begin{proposition} \label{prop:equivalence}
Suppose that $\tau \in W^{1, \infty}((0, L); S^{n - 1})$. Let $\lambda \in \R^n$ and $k \ge 0$.
\begin{enumerate}
\item
Suppose that $u \in W^{1, \infty}((0, L); \R^n) \setminus \{0\}$ satisfies \eqref{eqn:system1}
and \eqref{eqn:system2} almost everywhere in $(0, L)$. Then there exists
$f \in W^{1, \infty}(0, L) \setminus \{0\}$ with $f \ge 0$ such that
\begin{align}
f(\tau'' + k^2 \tau) & = \beta k^2 \proj_{\tau, \tau'}^\perp(\lambda), \label{eqn:ODE_rescaled1} \\
f' & = \beta \lambda \cdot \tau' \label{eqn:ODE_rescaled2},
\end{align}
weakly in $(0, L)$. If $k > 0$, then $f = k|u|$ has this property.
\item
Suppose that there exists
$f \in W^{1, \infty}(0, L) \setminus \{0\}$ with $f \ge 0$ satisfying \eqref{eqn:ODE_rescaled1} and \eqref{eqn:ODE_rescaled2}
weakly in $(0, L)$. Then there exists $u \in W^{1, \infty}((0, L); \R^n) \setminus \{0\}$ such that \eqref{eqn:system1}
and \eqref{eqn:system2} hold almost everywhere; and if $k > 0$, such that also $f = k|u|$.
\item
If there exists $f \in W^{1, \infty}(0, L) \setminus \{0\}$ with $f \ge 0$ such that
\eqref{eqn:ODE_rescaled1} and \eqref{eqn:ODE_rescaled2} hold weakly and $f + \beta \lambda \cdot \tau \le 0$
in $(0, L)$, then $\tau$ minimises $K_\infty$ subject to the constraints \eqref{eqn:boundary_condition1} and \eqref{eqn:boundary_condition2}.
\end{enumerate}
\end{proposition}

\begin{proof}
Suppose first that we have a weak solution of \eqref{eqn:system1} and \eqref{eqn:system2}
for some $u \in W^{1, \infty}((0, L); \R^n) \setminus \{0\}$. Let $\Omega = \set{t \in [0, L]}{u(t) \neq 0}$.

If $k = 0$, then $\tau' = 0$ in $\Omega$ by \eqref{eqn:system2}. With the same arguments as in
the proof of Proposition \ref{prop:DE=>pseudo-minimiser}, we show that $\tau' = 0$ almost everywhere in $[0, L] \setminus \Omega$.
Hence \eqref{eqn:ODE_rescaled1}, \eqref{eqn:ODE_rescaled2} automatically hold true for any constant
function $f$.

If $k > 0$, then we consider the function $f = k|u|$.
Equation \eqref{eqn:system2} then implies that $u = f\tau'/k^2$ almost everywhere.
We conclude that $\tau' = k^2 u/f$ in $\Omega$, so $\tau \in W_\loc^{2, \infty}(\Omega)$.
Hence from \eqref{eqn:system1} we derive the equation
\begin{equation} \label{eqn:DE4}
f(\tau'' + |\tau'|^2 \tau) + f' \tau' = k^2\beta (\lambda - (\lambda \cdot \tau)\tau)
\end{equation}
almost everywhere in $\Omega$.
Taking the inner product with $\tau'$ and observing that $\tau \cdot \tau' = 0$ (because $|\tau| \equiv 1$)
and $\tau'' \cdot \tau' = 0$ (because $|\tau'| \equiv k$ in $\Omega$), we see that
\[
k^2f' = k^2 \beta \lambda \cdot \tau'.
\]
This amounts to equation \eqref{eqn:ODE_rescaled2}.
Of course $f \ge 0$ by the definition of $f$.

Differentiating the equation $\tau \cdot \tau' = 0$, we see that $\tau \cdot \tau'' + |\tau'|^2 = 0$.
Recalling that $\tau' \cdot \tau'' = 0$, we conclude that
\[
\proj_{\tau, \tau'}^\perp(\tau'') = \tau'' - (\tau \cdot \tau'') \tau = \tau'' + k^2 \tau
\]
in $\Omega$.
Applying $\proj_{\tau, \tau'}^\perp$ to both sides of
\eqref{eqn:DE4}, we see that \eqref{eqn:ODE_rescaled1} holds almost everywhere in $\Omega$.
Also note that the function $f\tau' = k^2 u$ is continuous. Thus if $(t_1, t_2)$
is any connected component of $\Omega$, then for any $\xi \in C_0^\infty((0, L); \R^n)$,
\begin{multline*}
\int_{t_1}^{t_2} \left(f(\tau' \cdot \xi' - k^2 \tau \cdot \xi) + f'\tau' \cdot \xi + k^2 \beta \proj_{\tau, \tau'}^\perp(\lambda) \cdot \xi\right) \, dt \\
= k^2u(t_2) \cdot \xi(t_2) - k^2u(t_1) \cdot \xi(t_1) = 0.
\end{multline*}
A similar conclusion holds if $[0, t_2)$ or $(t_1, L]$ is a connected component of $\Omega$.
Away from $\Omega$, we know that $u = 0$ and therefore either $\lambda = 0$ or $\tau = \pm \lambda/|\lambda|$
almost everywhere in $[0, L] \setminus \Omega$ by \eqref{eqn:system1}. Hence \eqref{eqn:ODE_rescaled1}
holds weakly in all of $(0, L)$.

Conversely, suppose that we have a weak solution of \eqref{eqn:ODE_rescaled1}, \eqref{eqn:ODE_rescaled2}
for $f \in W^{1, \infty}(0, L) \setminus \{0\}$ with $f \ge 0$.
Consider the open set $\Omega = \set{t \in [0, L]}{f(t) \neq 0}$.
Here we can use \eqref{eqn:ODE_rescaled1} to conclude that $\tau \in W_\loc^{2, \infty}(\Omega)$.
We differentiate the equation $|\tau|^2 = 1$ twice
and we obtain $\tau'' \cdot \tau + |\tau'|^2 = 0$ almost everywhere
in $\Omega$. On the other hand, multiplying both sides of \eqref{eqn:ODE_rescaled1}
with $\tau$, we find that $\tau'' \cdot \tau + k^2 = 0$ in $\Omega$. Hence $|\tau'| \equiv k$ in $\Omega$.

If $k = 0$, then $\tau' \equiv 0$ in $\Omega$ and \eqref{eqn:ODE_rescaled2} implies that $f$ is
locally constant in $\Omega$. So in this case, it follows that $\Omega = [0, L]$
and \eqref{eqn:system2} is automatically satisfied. Moreover, it is then easy to find
$u \in W^{1, \infty}((0, L); \R^n) \setminus \{0\}$ that solves \eqref{eqn:system1}.

If $k > 0$, then we claim that \eqref{eqn:DE4} is satisfied in $\Omega$. In order to see why, we split the equation
into three parts by projecting orthogonally onto the spaces $\R \tau(t)$ and $\R \tau'(t)$ and onto the
orthogonal complement of $\R \tau(t) \oplus \R \tau'(t)$ at almost every $t \in \Omega$.
The projection onto $\R \tau(t)$ is trivial. The projection onto $\R \tau'(t)$ amounts to
\eqref{eqn:ODE_rescaled2}, and applying $\proj_{\tau(t), \tau'(t)}^\perp$ gives \eqref{eqn:ODE_rescaled1}.
Thus we have a solution of \eqref{eqn:DE4} in $\Omega$.

Setting $u = f\tau'/k^2$,
we can then verify \eqref{eqn:system1} and \eqref{eqn:system2} in $\Omega$.
Outside of $\Omega$, we know that $f = 0$ and $u = 0$. Hence \eqref{eqn:ODE_rescaled2} implies that
$\lambda \cdot \tau' = 0$ almost everywhere outside of $\Omega$.
Moreover, \eqref{eqn:ODE_rescaled1} implies that $\proj_{\tau, \tau'}^\perp(\lambda) = 0$
almost everywhere in $[0, L] \setminus \Omega$.
That is, $\lambda$ is a multiple of $\tau$ and \eqref{eqn:system1}, \eqref{eqn:system2} are satisfied almost everywhere
in $[0, L] \setminus \Omega$ as well.

Furthermore, if $f + \beta \lambda \cdot \tau \le 0$, then $k|u| + \beta \lambda \cdot \tau \le 0$, and
the last statement follows from Proposition \ref{prop:DE=>pseudo-minimiser}.
\end{proof}

As mentioned previously, the new system of differential equations \eqref{eqn:ODE_rescaled1}, \eqref{eqn:ODE_rescaled2}
corresponds to \eqref{eqn:ODE1}, \eqref{eqn:ODE2} up to the reparametrisation from Section
\ref{sect:reparametrisation}. But Proposition \ref{prop:equivalence} requires only that $\lambda \in \R^n$,
whereas $\lambda \in S^{n - 1}$ in Theorem~\ref{thm:DE}. For this reason, the following
observation is useful.

\begin{lemma} \label{lem:lambda}
Let $\tau \in W^{1, \infty}((0, L); S^{n - 1})$, $\lambda \in \R^n$, and $f \in W^{1, \infty}(0, L) \setminus \{0\}$
with $f \ge 0$ such that \eqref{eqn:ODE_rescaled1}, \eqref{eqn:ODE_rescaled2}
hold weakly. Then there exist $\tilde{f} \in W^{1, \infty}(0, L) \setminus \{0\}$
with $\tilde{f} \ge 0$ and $\tilde{\lambda} \in S^{n - 1}$ such that \eqref{eqn:ODE_rescaled1}, \eqref{eqn:ODE_rescaled2}
hold weakly for $\tilde{f}$ instead of $f$ and for $\tilde{\lambda}$ instead of $\lambda$ as well.
\end{lemma}

\begin{proof}
If $\lambda \neq 0$, then it suffices to define $\tilde{f} = f/|\lambda|$ and $\tilde{\lambda} = \lambda/|\lambda|$
and check that both equations are still satisfied.
If $\lambda = 0$, then $f$ is constant and positive. Hence $\tau'' + k^2 \tau = 0$ in $(0, L)$.
With the same arguments as in the proof of Proposition \ref{prop:equivalence}, we see
that $|\tau'| \equiv k$. The resulting equation $\tau'' + |\tau'|^2 \tau = 0$ means that
$\tau$ follows a geodesic, i.e., a great circle on $S^{n - 1}$.
This implies that $\tau(t)$ and $\tau'(t)$ span the same two-dimensional
subspace of $\R^n$ everywhere, and any $\tilde{\lambda}$ in this subspace will satisfy
$\proj_{\tau, \tau'}^\perp(\tilde{\lambda}) = 0$. Now we choose $\tilde{f}$ such
that \eqref{eqn:ODE_rescaled2} holds true (for $\tilde{\lambda}$ instead of $\lambda$)
and at the same time $\tilde{f} > 0$ in $[0, L]$. Then both equations
are satisfied.
\end{proof}

We now have all the tools for the proofs of the first two results in the
introduction.

\begin{proof}[Proofs of Theorem \ref{thm:DE} and Theorem \ref{thm:minimiser}]
With the reparametrisation from Section \ref{sect:reparametrisation}, an $\infty$-elastica
gives rise to a pseudo-minimiser of $K_\infty$ and vice versa.
According to Proposition \ref{prop:pseudo-minimiser=>DE} and Proposition \ref{prop:DE=>pseudo-minimiser},
pseudo-minimisers of $K_\infty$ correspond to solutions of \eqref{eqn:system1}, \eqref{eqn:system2},
which is equivalent to \eqref{eqn:ODE_rescaled1}, \eqref{eqn:ODE_rescaled2} by Proposition \ref{prop:equivalence}.
Lemma \ref{lem:lambda} shows that it suffices to consider this system for $\lambda \in S^{n - 1}$.
Now we check that the system corresponds to \eqref{eqn:ODE1}, \eqref{eqn:ODE2} for
the original parametrisation, and this proves Theorem \ref{thm:DE}.
Theorem \ref{thm:minimiser} follows from the last statement of Proposition \ref{prop:equivalence}.
\end{proof}

\section{Preparation for the proof of Theorem \ref{thm:structure}}

The system of ordinary differential equations \eqref{eqn:ODE_rescaled1}, \eqref{eqn:ODE_rescaled2}
becomes degenerate at points where $f$ vanishes. It turns out, however,
that $f$ remains positive for \emph{generic} solutions as described in the following result.
This information will be crucial for statement \ref{item:3D} in Theorem \ref{thm:structure}.

\begin{lemma} \label{lem:ODE}
Let $\lambda, \tau_0 \in S^{n - 1}$ and $\tau_1 \in \R^n$
such that $\tau_0 \perp \tau_1$, and let $f_0 > 0$ and $t_0 \in [0, L]$. If the vectors $\tau_0$, $\tau_1$, and $\lambda$
are linearly independent, then the initial value problem
\begin{align}
\tau'' + |\tau'|^2 \tau & = \beta f^{-1} |\tau'|^2 \proj_{\tau, \tau'}^\perp(\lambda), \label{eqn:ODE_reformulated}  \\
f' & = \beta \lambda \cdot \tau', \nonumber\\
\tau(t_0) & = \tau_0, \quad \tau'(t_0) = \tau_1, \quad f(t_0) = f_0, \nonumber
\end{align}
has a unique global solution, consisting of $\tau \colon [0, L] \to S^{n - 1}$ and $f \colon [0, L] \to (0, \infty)$.
For all $t \in [0, L]$, this solution satisfies $|\tau'(t)| = |\tau_1|$ and $\lambda \cdot \tau(t) \neq \pm 1$,
and $\tau(t)$ remains in the linear subspace of $\R^n$ spanned by
$\tau_0$, $\tau_1$, and $\lambda$.
\end{lemma}

\begin{proof}
Under these assumptions, we clearly have a unique solution of the initial value problem
in a certain interval $(t_1, t_2) \cap [0, L]$ such that $\lambda \cdot \tau \neq \pm 1$ and $f > 0$
in that interval.
Multiplying \eqref{eqn:ODE_reformulated} with $\tau$, we see that $\frac{d}{dt}(\tau \cdot \tau') = 0$. Hence the
solution will continue to take values on the sphere $S^2$. Multiplying the equation with $\tau'$,
we further see that $\frac{d}{dt} |\tau'|^2 = 0$. Setting $k = |\tau_1|$, we conclude that
$|\tau'| = k$ in $(t_1, t_2) \cap [0, L]$.
Moreover, if $V \in \R^n$ is any vector perpendicular to $\tau_0$, $\tau_1$, and $\lambda$,
then the function $h = V \cdot \tau$ satisfies
\[
h'' + |\tau'|^2 h = -\beta f^{-1} \left(|\tau'|^2 (\lambda \cdot \tau) h + (\lambda \cdot \tau') h'\right)
\]
in $(t_1, t_2) \cap [0, L]$
and $h(t_0) = h'(t_0) = 0$. Hence $h \equiv 0$, and the solution $\tau$ will
remain in the linear subspace spanned by $\tau_0$, $\tau_1$, and $\lambda$ in $(t_1, t_2) \cap [0, L]$. So we may assume
that $n = 3$ without loss of generality. We may further choose coordinates such that $\lambda = (0, 0, 1)$.

It now suffices to show that $\liminf_{t \searrow t_1} f(t) > 0$ and $\limsup_{t \searrow t_1} |\lambda \cdot \tau(t)| < 1$
(unless $t_1 < 0$) and that $\liminf_{t \nearrow t_2} f(t) > 0$ and $\limsup_{t \nearrow t_2} |\lambda \cdot \tau(t)| < 1$
(unless $t_2 > L$).
The standard theory for ordinary differential equations
will then imply the result.

We use spherical coordinates on $S^2$ and we write
\[
\tau = (\cos \varphi \sin \vartheta, \sin \varphi \sin \vartheta, \cos \vartheta)
\]
for $\varphi, \vartheta \colon (t_1, t_2) \cap [0, L] \to \R$ with $\vartheta(t) \in (0, \pi)$ for all $t \in (t_1, t_2)$.
Writing also
\[
e_1 = (-\sin \varphi, \cos \varphi, 0) \quad \text{and} \quad e_2 = (\cos \varphi \cos \vartheta, \sin \varphi \cos \vartheta, -\sin \vartheta),
\]
we obtain an orthonormal basis $(\tau(t), e_1(t), e_2(t))$ of $\R^3$ such that $e_1(t)$ and $e_2(t)$ span
the tangent space of $S^2$ at $\tau(t)$ for every $t \in (t_1, t_2)\cap [0, L]$. We compute
\[
\tau' = \varphi' \sin \vartheta \, e_1 + \vartheta' \, e_2
\]
and
\[
\tau'' + |\tau'|^2 \tau = \left(\varphi'' \sin \vartheta + 2\varphi' \vartheta' \cos \vartheta\right) e_1 + \left(\vartheta'' - (\varphi')^2 \sin \vartheta \cos \vartheta\right) e_2.
\]
Define $Z = -\vartheta' \, e_1 + \varphi' \sin \vartheta \, e_2$, so that $|Z| = |\tau'|$ and $Z \perp \tau'$. Then
\[
|\tau'|^2 \proj_{\tau, \tau'}^\perp(\lambda) = (\lambda \cdot Z)Z = \varphi' \vartheta' \sin^2 \vartheta \, e_1 - (\varphi')^2 \sin^3 \vartheta \, e_2.
\]
Therefore, we obtain the equations
\begin{align}
\varphi'' \sin \vartheta + 2\varphi' \vartheta' \cos \vartheta & = \beta f^{-1} \varphi'\vartheta' \sin^2 \vartheta, \label{eqn:varphi} \\
\vartheta'' - (\varphi')^2 \sin \vartheta \cos \vartheta & = - \beta f^{-1} (\varphi')^2 \sin^3 \vartheta, \label{eqn:vartheta}
\end{align}
and furthermore
\begin{equation} \label{eqn:f_polar}
f' = - \beta \vartheta' \sin \vartheta.
\end{equation}
For the rest of the proof, it suffices to consider \eqref{eqn:varphi} and \eqref{eqn:f_polar}.

We first claim that $\varphi'$ does not vanish anywhere in $(t_1, t_2) \cap [0, L]$. Otherwise, equation
\eqref{eqn:varphi} would imply that it remains $0$ throughout $(t_1, t_2) \cap [0, L]$,
and $\tau$ would parametrise a piece of a great circle through $(0, 0, 1)$. This, however, is impossible
under the assumption that $\tau_0$, $\tau_1$, and $\lambda$ are linearly independent.

Thus we may divide by $\varphi' \sin \vartheta$ in \eqref{eqn:varphi} and we find that
\[
\frac{\varphi''}{\varphi'} = -\frac{f'}{f} - \frac{2\vartheta' \cos \vartheta}{\sin \vartheta}.
\]
Integrating, we see that there exists $b \in \R$ such that
\[
\log |\varphi'| = - \log f - 2 \log \sin \vartheta + b.
\]
Set $B = e^b$. Then
\[
|\varphi'| = \frac{B}{f \sin^2 \vartheta}.
\]
The equation $(\varphi')^2 \sin^2 \vartheta + (\vartheta')^2 = |\tau'|^2 = k^2$ then implies that
\[
\frac{B^2}{f^2 \sin^2 \vartheta} \le k^2.
\]
It follows immediately that $f$ and $\sin \vartheta$ stay away from $0$ and this concludes the proof.
\end{proof}

The following technical lemma is also required for the proof of Theorem \ref{thm:structure}.

\begin{lemma} \label{lem:series}
Suppose that $(b_i)_{i \in \N}$ is a sequence of positive numbers such that
\[
\sum_{i = 1}^\infty \left|1 - \frac{b_{i + 1}}{b_i}\right| < \infty.
\]
Then $\sum_{i = 1}^\infty b_i = \infty$.
\end{lemma}

\begin{proof}
Ignoring finitely many terms if necessary, we may assume that
\[
\sum_{i = 1}^\infty \left|1 - \frac{b_{i + 1}}{b_i}\right| \le \frac{1}{2}.
\]

Fix $I \in \N$.
Let $q_i = b_{i + 1}/b_i$ for $i = 1, \dotsc, I - 1$. Choose a permutation $S \colon \{1, \dots, I - 1\} \to \{1, \dots, I - 1\}$
such that $q_{S(1)} \le \dotsb \le q_{S(I - 1)}$ and define $q_i' = \min\{q_{S(i)}, 1\}$.
Also define $b_1', \dotsc, b_I' > 0$ by $b_1' = b_1$ and
\[
b_{i + 1}' = q_i' b_i', \quad i = 1, \dotsc, I - 1.
\]
Then
\begin{equation} \label{eqn:quotients}
\sum_{i = 1}^{I - 1} (1 - q_i') \le \sum_{i = 1}^{I - 1} |1 - q_i| \le \frac{1}{2}
\end{equation}
and $b_i' \le b_i$ for all $i = 1, \dots, I$.

As $q_i'$ is non-decreasing in $i$, inequality \eqref{eqn:quotients} implies that
\[
1 - q_i' \le \frac{1}{2i} \le \frac{1}{i + 1}
\]
for $i = 1, \dots, I - 1$. Define $B_i = 1/i$ for $i = 1, \dotsc, I$. Then
\[
\frac{b_{i + 1}'}{b_i'} = q_i' \ge \frac{i}{i + 1} = \frac{B_{i + 1}}{B_i}.
\]
Hence
\[
b_i' = \frac{b_i'}{b_{i - 1}'} \dotsb \frac{b_2'}{b_1'} b_1 \ge \frac{B_i}{B_{i - 1}} \dotsb \frac{B_2}{B_1} b_1 = \frac{B_i}{B_1} b_1 = \frac{b_1}{i}.
\]
It follows that
\[
\sum_{i = 1}^I b_i \ge \sum_{i = 1}^I b_i' \ge b_1 \sum_{i = 1}^I \frac{1}{i}.
\]
Letting $I \to \infty$, we obtain the desired result.
\end{proof}

\section{Proof of Theorem \ref{thm:structure}}

Now we consider the situation of Theorem \ref{thm:structure}.
Suppose first that $\gamma \in \mathcal{G}$ is an $\infty$-elastica and let $k = \K_\alpha(\gamma)$.
If $k = 0$, then $\gamma'' = 0$ almost everywhere and $\gamma$ parametrises
a line segment. Then clearly statement \ref{item:2D} in Theorem \ref{thm:structure}
is satisfied. Therefore, we assume that $k > 0$ henceforth.

Consider the
reparametrised tangent vector field $\tau \colon [0, L] \to S^{n - 1}$ with $\tau(t) = \gamma'(\phi(t))$
for $t \in [0, L]$ as in Section \ref{sect:reparametrisation}.
Then $\tau$ is a pseudo-minimiser of $K_\infty$. Hence by Proposition \ref{prop:pseudo-minimiser=>DE},
there exist $\lambda \in \R^n$ and $u \in W^{1, \infty}((0, L); \R^n) \setminus \{0\}$ such that
\eqref{eqn:system1} and \eqref{eqn:system2}
hold true almost everywhere. According to Proposition \ref{prop:equivalence}, the function
$f = k|u|$ satisfies \eqref{eqn:ODE_rescaled1} and \eqref{eqn:ODE_rescaled2} weakly, and by Lemma \ref{lem:lambda}
we may assume that $\lambda \in S^{n - 1}$.

Let $\Omega = \set{t \in [0, L]}{f(t) > 0}$. Then \eqref{eqn:system2}
implies that $\tau'$ is continuous in $\Omega$ with $|\tau'| \equiv k$.
It follows from \eqref{eqn:ODE_rescaled1} that $\tau \in W_\loc^{2, \infty}(\Omega)$. Moreover, by
standard theory for ordinary differential equations, both $\tau$ and $f$ are locally uniquely determined by their initial conditions
$\tau(t_0)$, $\tau'(t_0)$, and $f(t_0)$ for any $t_0 \in \Omega$.

If $\tau$, $\tau'$, and $\lambda$ are linearly independent anywhere in $\Omega$, then
Lemma \ref{lem:ODE} implies that $\Omega = [0, L]$ and that $\tau$ takes values in a three-dimensional subspace of $\R^n$,
and \eqref{eqn:ODE_rescaled1} and \eqref{eqn:ODE_rescaled2} are satisfied almost everywhere.
Equations \eqref{eqn:ODE1} and \eqref{eqn:ODE2}
now arise when we reverse the reparametrisation from Section \ref{sect:reparametrisation}.
The observation that $\alpha \gamma'' = \tau' \circ \psi$ implies that $\alpha \gamma'' \in W^{1, \infty}((0, \ell); \R^n)$
and that $\alpha |\gamma''| \equiv k$.
Equation \eqref{eqn:ODE2} then implies that $g \in W^{2, \infty}(0, \ell)$. Hence
statement \ref{item:3D} in Theorem \ref{thm:structure} holds true.

This leaves the case when $\tau$, $\tau'$, and $\lambda$ are linearly dependent everywhere in $\Omega$.
We assume this from now on. Then we can say more about
the behaviour of $\tau$ in $\Omega$.

\begin{lemma} \label{lem:Omega}
If $(t_1, t_2) \subseteq \Omega$, then
the restriction of $\tau$ to $(t_1, t_2)$ follows a great circle in $S^{n - 1}$
through $\lambda$ with constant speed $k$. Furthermore, if $(t_1, t_2)$ is a connected component of $\Omega$,
then there exists $t_0 \in (t_1, t_2)$
such that $\tau(t_0) = \pm \lambda$.
\end{lemma}

\begin{proof}
We know that $\tau \cdot \tau' = 0$ everywhere, and $\tau'$ is continuous with $|\tau'| \equiv k$ in $\Omega$.
As $\tau(t)$, $\tau'(t)$, and $\lambda$
are linearly dependent, we further know that $\tau'(t)$ is in the space spanned by $\tau(t)$ and $\lambda$
for every $t \in \Omega$ with $\tau(t) \neq \pm \lambda$. Hence
$\tau$ follows a great circle on $S^{n - 1}$ through $\lambda$ with speed $k$; indeed, by the continuity of $\tau'$,
this is true throughout $(t_1, t_2)$ even if there are any points where $\tau(t) = \pm \lambda$.
If $(t_1, t_2)$ is a connected component of $\Omega$, then $f(t_1) = 0 = f(t_2)$. By
\eqref{eqn:ODE_rescaled2}, this means that $\lambda \cdot \tau'$ must change sign
somewhere in $(t_1, t_2)$. Given what we know about $\tau$ so far, there must exists $t_0 \in (t_1, t_2)$ such that $\tau(t_0) = \pm \lambda$.
\end{proof}

Next consider the set $\Omega' = \set{t \in [0, L]}{\tau(t) \neq \pm \lambda} \cup \Omega$. This is an open
set relative to $[0, L]$ as well.

\begin{lemma} \label{lem:discrete}
The set $\Omega' \setminus \Omega$ is discrete.
\end{lemma}

\begin{proof}
As $f = 0$ in $[0, L] \setminus \Omega$, we know that $f' = 0$ almost everywhere in this set.
Using \eqref{eqn:ODE_rescaled2}, we conclude that $\tau' \cdot \lambda = 0$ almost everywhere,
and \eqref{eqn:ODE_rescaled1} implies that $\lambda$ is in the subspace spanned by $\tau$ and $\tau'$
almost everywhere in $[0, L] \setminus \Omega$. Hence $\tau = \pm \lambda$ almost
everywhere in $[0, L] \setminus \Omega$.
It follows that $\Omega' \setminus \Omega$ is a null set, and so is $\Omega' \setminus \overline{\Omega}$.
As the latter is an open set, it must be empty.
So $\Omega' \subseteq \overline{\Omega}$.

For any $t_0 \in \Omega' \setminus \Omega$,
we may choose $\epsilon > 0$ such that $\tau \neq \pm \lambda$ in $(t_0 - \epsilon, t_0 + \epsilon) \cap [0, L]$
by the continuity of $\tau$. Let $J = (t_0 - \epsilon, t_0 + \epsilon) \cap (0, L)$. Then
$J$ cannot contain any connected components of $\Omega$
by Lemma \ref{lem:Omega}. Therefore, the open set $J \cap \Omega$ consists of at most two intervals
extending to one of the end points of $J$. But we know that $J \subseteq \overline{\Omega}$.
Hence $J \cap \Omega = J \setminus \{t_0\}$.
We conclude that $t_0$ is an isolated point of $\Omega' \setminus \Omega$.
That is, the set $\Omega' \setminus \Omega$ is discrete.
\end{proof}

\begin{lemma} \label{lem:Omega'}
If $I$ is any connected component
of $\Omega'$, then the restriction of $\tau$ to $I$ takes values in a
great circle on $S^{n - 1}$ through $\lambda$.
\end{lemma}

\begin{proof}
In view of Lemma \ref{lem:Omega} and Lemma \ref{lem:discrete}, it suffices to examine
what happens near a point $t_0 \in I \setminus \Omega$. There exists $\epsilon > 0$
such that the restriction of $\tau$ to $(t_0 - \epsilon, t_0)$ follows a great circle through
$\lambda$, and the same statement applies to $(t_0, t_0 + \epsilon)$. But as $t_0 \in I \subseteq \Omega'$
and $t_0 \not\in \Omega$, it is clear that $\tau(t_0) \neq \pm \lambda$. So we have the
same great circle on both sides of $t_0$, and the claim follows.
\end{proof}

We can now improve Lemma \ref{lem:discrete}. This is the only place in the paper where
we use the assumption that $\alpha$ is of bounded variation rather than just bounded.

\begin{lemma} \label{lem:finite}
If $I \subseteq \Omega'$ is a connected component of $\Omega'$, then $I \setminus \Omega$ is finite.
\end{lemma}

\begin{proof}
We argue by contradiction here, so we assume that $I \setminus \Omega$ is \emph{not} finite.
Then by Lemma \ref{lem:discrete}, either $\inf I$ or $\sup I$ is an accumulation point of $I \setminus \Omega$, and
we assume for simplicity that this is true for $\sup I$. (The arguments are similar if
it is $\inf I$.) Then there is a sequence $(t_i)_{i \in \N}$ in $I \setminus \Omega$
such that $t_{i + 1} > t_i$ and $(t_i, t_{i + 1}) \subseteq \Omega$ for all $i \in \N$.
So $f(t_i) = 0$ for all $i \in \N$. By Lemma \ref{lem:Omega}, we know that $\tau$
follows a great circle through $\lambda$ with speed $k$ in the interval $(t_i, t_{i + 1})$ and
there exists a point $\rho_i \in (t_i, t_{i + 1})$
such that $\tau(\rho_i) = \pm \lambda$ for every $i \in \N$. If $\tau(\rho_i) = \lambda$
and $\tau(\rho_{i + 1}) = - \lambda$ or vice versa, then $\rho_{i + 1} - \rho_i \ge \pi/k$;
so this can happen at most a finite number of times. Dropping finitely many members of the
sequence, we may assume that $\rho_{i + 1} - \rho_i < \pi/k$ for every $i$; then $\tau(\rho_i)$ has always the same
sign and for simplicity we
assume that $\tau(\rho_i) = \lambda$ for every $i \in \N$. Then
\[
\lambda \cdot \tau(t) = \cos(k(t -\rho_i))
\]
in $(t_i, t_{i + 1})$ for all $i \in \N$.

It follows immediately that $\rho_{i + 1} - t_{i + 1} = t_{i + 1} - \rho_i$ for every $i \in \N$.
Furthermore, equation \eqref{eqn:ODE_rescaled2} implies that
\[
0 = \int_{t_i}^{t_{i + 1}} f'(t) \, dt = - k \int_{t_i}^{t_{i + 1}} \beta(t) \sin(k(t -\rho_i)) \, dt.
\]
Hence
\[
k \int_{t_i}^{\rho_i} \beta(t) \left|\sin(k(t -\rho_i))\right| \, dt = k \int_{\rho_i}^{t_{i + 1}} \beta(t) \left|\sin(k(t -\rho_i))\right| \, dt.
\]
Define
\[
b_i = k \int_{t_i}^{\rho_i}  \left|\sin(k(t -\rho_i))\right| \, dt = 1 - \cos(k(\rho_i - t_i))
\]
and
\[
b_i' = k \int_{\rho_i}^{t_{i + 1}} \left|\sin(k(t -\rho_i))\right| \, dt = 1 - \cos(k(t_{i + 1} - \rho_i)).
\]
If $b_i' \le b_i$, then we may choose $\omega_i \in [t_i, \rho_i]$ and $\omega_i' \in [\rho_i, t_{i + 1}]$ such
that
\[
k \int_{t_i}^{\rho_i} \beta(t) \left|\sin(k(t -\rho_i))\right| \, dt \ge b_i \beta(\omega_i)
\]
and
\[
k \int_{\rho_i}^{t_{i + 1}} \beta(t) \left|\sin(k(t -\rho_i))\right| \, dt \le b_i' \beta(\omega_i');
\]
then
\[
\frac{b_i'}{b_i} \ge \frac{\beta(\omega_i)}{\beta(\omega_i')}.
\]
If $b_i < b_i'$, then instead we choose $\omega_i \in [t_i, \rho_i]$ and $\omega_i' \in [\rho_i, t_{i + 1}]$ such
that
\[
k \int_{t_i}^{\rho_i} \beta(t) \left|\sin(k(t -\rho_i))\right| \, dt \le b_i \beta(\omega_i)
\]
and
\[
k \int_{\rho_i}^{t_{i + 1}} \beta(t) \left|\sin(k(t -\rho_i))\right| \, dt \ge b_i' \beta(\omega_i');
\]
then
\[
\frac{b_i'}{b_i} \le \frac{\beta(\omega_i)}{\beta(\omega_i')}.
\]
In both cases,
\[
\left|1 - \frac{b_i'}{b_i}\right| \le \left|1 - \frac{\beta(\omega_i)}{\beta(\omega_i')}\right| = \frac{|\beta(\omega_i') - \beta(\omega_i)|}{|\beta(\omega_i')|} \le |\beta(\omega_i') - \beta(\omega_i)| \, \sup_{[0, \ell]} \frac{1}{\alpha}.
\]
Hence
\[
\sum_{i = 1}^\infty \left|1 - \frac{b_i'}{b_i}\right| \le \sup\biggl\{\sum_{j = 1}^J |\alpha(s_j) - \alpha(s_{j - 1})| \colon 0 \le s_0 \le \dotsb \le s_J \le \ell\biggr\} \, \sup_{[0, \ell]} \frac{1}{\alpha}.
\]
The right-hand side is finite, because $\alpha$ is assumed to be of bounded variation and $1/\alpha$ is bounded.

We have already seen that $t_{i + 1} - \rho_i = \rho_{i + 1} - t_{i + 1}$ for every $i \in \N$.
This means that $b_i' = b_{i + 1}$. We now apply Lemma \ref{lem:series} to the sequence
$(b_1, b_1', b_2, b_2', \dotsc)$. We infer that
\begin{equation} \label{eqn:divergent_sum}
\sum_{i = 1}^\infty (b_i + b_i') = \infty.
\end{equation}
But clearly
\[
\sum_{i = 1}^\infty (\rho_i - t_i) + \sum_{i = 1}^\infty (t_{i + 1} - \rho_i)  \le L,
\]
as this is the sum of the lengths of pairwise disjoint intervals in $(0, L)$.
Hence there exists $i_0 \in \N$ such that
\[
\rho_i - t_i \le \frac{2}{k^2} \quad \text{and} \quad t_{i + 1} - \rho_i \le \frac{2}{k^2}
\]
for all $i \ge i_0$, which implies that
\[
b_i = 1 - \cos(k(\rho_i - t_i)) \le \rho_i - t_i
\]
and
\[
b_i' = 1 - \cos(k(t_{i + 1} - \rho_i)) \le t_{i + 1} - \rho_i.
\]
Now we have a contradiction to \eqref{eqn:divergent_sum}.
\end{proof}

\begin{lemma} \label{lem:finite2}
The set $\Omega'$ has finitely many connected components.
\end{lemma}

\begin{proof}
We can ignore any connected components of the form $[0, t_2)$ or $(t_1, L]$. Thus we fix
another connected component $I = (t_1, t_2)$.
Then $f(t_1) = 0$ and $\tau(t_1) = \pm \lambda$,
and also $f(t_2) = 0$ and $\tau(t_2) = \pm \lambda$. Furthermore, by Lemma \ref{lem:finite},
there exists $t_3 \in (t_1, t_2]$ such that $f(t_3) = 0$ and $(t_1, t_3) \subseteq \Omega$.
According to Lemma \ref{lem:Omega}, this implies that there exists $t_4 \in (t_1, t_3)$
with $\tau(t_4) = \pm \lambda$. We further know that $\tau$ follows a great circle with speed $k$ in
$(t_1, t_4)$, and therefore $t_4 - t_1 \ge \pi/k$. So there can only
be finitely many connected components.
\end{proof}

Now we can complete the proof of Theorem \ref{thm:structure} as follows.

By Lemma \ref{lem:finite2}, we can partition $\Omega'$ into finitely many connected components $I_1, \dotsc, I_M$.
Let $t_i = \inf I_i$ and $t_i' = \sup I_i$ for $i = 1, \dotsc, M$.
Setting $A = [0, L] \setminus \bigcup_{i = 1}^M I_i$, we observe that $f = 0$ and
$\tau = \pm \lambda$ on $A$.

The set $\tau(\overline{I}_i)$ is contained in a two-dimensional subspace $X_i \subseteq \R^n$
with $\lambda \in X_i$
for every $i = 1, \dotsc, M$ by Lemma \ref{lem:Omega'}.
Hence Lemma~\ref{lem:2D} may be applied to the restriction of $\tau$ to $\overline{I}_i$.
Consequently, there exists a line $\mathcal{L}_i \subseteq X_i + c(t_i)$
for every $i = 1, \dotsc, M$ such that $\set{t \in \overline{I}_i}{f(t) = 0} = \set{t \in \overline{I}_i}{c(t) \in \mathcal{L}_i}$,
where $c = \gamma \circ \phi$. But we know that
$f(t_i) = 0$, except possibly for $i = 1$ if $t_1 = 0$, and that
$f(t_i') = 0$, except possibly for $i = M$ if $t_M' = L$.
Moreover, each $\mathcal{L}_i$ is parallel to $\lambda$. As $\tau = \pm \lambda$ on $A$, we
also conclude that $c([t_i', t_{i + 1}])$ is a line segment parallel
to $\lambda$ for $i = 1, \dotsc, M - 1$, and the same applies to $c([0, t_1])$ if $t_1 > 0$ and
to $c([t_M', L])$ if $t_M' < L$.
Hence the lines $\mathcal{L}_i$ all coincide with a single
line $\mathcal{L} \subseteq \R^n$ and $c(A) \subseteq \mathcal{L}$.

If there are any points $t \in I_i \setminus \Omega$, then we further subdivide $I_i$. According to Lemma \ref{lem:finite},
there are only finitely many such points.
Thus we obtain pairwise disjoint, relatively open intervals $I_1^*, \ldots, I_N^* \subseteq [0, L]$
such that $c(t) \not\in \mathcal{L}$ for all $t \in I_i^*$ for $i = 1, \dotsc, N$
but $c(t) \in \mathcal{L}$ for all $t \in [0, L] \setminus \bigcup_{i = 1}^N I_i^*$.
Lemma \ref{lem:2D} then further implies that $\tau'$ is continuous with $|\tau'| \equiv k$ in $I_i^*$,
and that there exists $\delta > 0$ such that for any $t_0 \in \overline{I_i^*} \setminus I_i^*$,
the inequality $\lambda \cdot \tau' > 0$ is satisfied in
$(t_0, t_0 + \delta) \cap I_i^*$ and $\lambda \cdot \tau' < 0$
in $(t_0 - \delta, t_0) \cap I_i^*$ for all $i = 1, \dotsc, N$.

Reversing the reparametrisation from Section \ref{sect:reparametrisation} and setting $J_i = \phi(I_i^*)$,
we therefore find the situation described in statement \ref{item:2D}
of Theorem \ref{thm:structure}.

Finally, we want to prove that every curve satisfying one of the conditions in Theorem \ref{thm:structure}
is indeed an $\infty$-elastica. This is clear if $\gamma([0, L])$ is contained in a line, so we assume
otherwise.

In the case of condition \ref{item:3D}, the claim follows immediately from Proposition \ref{prop:equivalence} and
Proposition \ref{prop:DE=>pseudo-minimiser}. If condition \ref{item:2D} is satisfied, we use
Lemma \ref{lem:2D} for any piece of $\gamma$ restricted to $\overline{J}_i$. In order to work with
the usual reparametrisation, we set $I_i = \psi(J_i)$ and let $t_i = \inf I_i$ and $t_i' = \sup I_i$.
Then Lemma \ref{lem:2D} gives rise to
$u_i \colon \overline{I}_i \to \R^n$ satisfying \eqref{eqn:system1}, \eqref{eqn:system2} in $I_i$ with
$u_i(t_i) = 0$ (unless $t_i = 0$) and $u_i(t_i') = 0$ (unless $t_i' = L$), but $u_i \neq 0$
in $I_i$. Hence we define
$u \colon [0, L] \to \R^n$ by
\[
u(t) = \begin{cases}
u_i(t) & \text{if } t \in I_i, \ i = 1, \dotsc, N, \\
0 & \text{else}.
\end{cases}
\]
Then \eqref{eqn:system1} and \eqref{eqn:system2} are satisfied almost
everywhere in $(0, L)$. Proposition \ref{prop:DE=>pseudo-minimiser}
now completes the proof.

\section{The Markov-Dubins problem} \label{sect:Dubins}

In this section, we first prove Proposition \ref{prop:shortest_curves}, thus establishing
the connection to the Markov-Dubins problem of minimising length subject to
curvature constraints. Then we show how to recover some of the main results of
Dubins \cite[Theorem I]{Dubins:57} and Sussmann \cite[Theorem 1]{Sussmann:95} from Theorem~\ref{thm:structure}.

\begin{proof}[Proof of Proposition \ref{prop:shortest_curves}]
Suppose that $\gamma \in \mathcal{G}$ does \emph{not} minimise $\K_1$ under the boundary conditions
\eqref{eqn:boundary_conditions}. We want to show that the curve parametrised by $\gamma$ is not an $R$-geodesic.
For $R > 1/\K_1(\gamma)$, this is obvious, as $\gamma$ does not satisfy the required curvature constraint.
Thus we assume that $R \le 1/\K_1(\gamma)$.

We may assume without loss of generality that
$a_1, a_2 \in \{0\}^{n - 1} \times \R$. In the following, we write $x = (x', x_n)$
for a generic point $x = (x_1, \dotsc, x_n) \in \R^n$, where $x' = (x_1, \dotsc, x_{n - 1})$.
Let $\epsilon > 0$ and consider the map $\Phi_\epsilon \colon \R^n \to \R^n$ defined by
\[
\Phi_\epsilon(x) = \left(\frac{x'}{1 + \epsilon |x'|^2}, x_n\right).
\]
This has the derivative $d\Phi_\epsilon(0, x_n) = \id_{\R^n}$ for any $x_n \in \R$.
We have the convergence $\Phi_\epsilon \to \id_{\R^n}$ in $C^2(C; \R^n)$
for any compact set $C \subseteq \R^n$ as $\epsilon \to 0$.
Moreover, for any $x, V \in \R^n$, unless $x' = 0$ or $V' = 0$, we find that $|d\Phi_\epsilon(x)V| < |V|$.
Now choose $\hat{\gamma} \in \mathcal{G}$ with $\K_1(\hat{\gamma}) < \K_1(\gamma)$.
Consider
$\hat{\gamma}_\epsilon = \Phi_\epsilon \circ \hat{\gamma}$ for some $\epsilon > 0$ that remains to be
determined. Then $\hat{\gamma}_\epsilon$ still satisfies the boundary conditions \eqref{eqn:boundary_conditions}.

As $\gamma$ does not minimise $\K_1$ by the above assumption,
we conclude that $\gamma([0, \ell]) \not\subseteq \{0\}^{n - 1} \times \R$.
Hence $|a_2 - a_1| < \ell$ and $\hat{\gamma}([0, \ell])$ is not contained in $\{0\}^{n - 1} \times \R$ either.
Therefore, the length of $\hat{\gamma}_\epsilon$ is strictly less than $\ell$.
But $\hat{\gamma}_\epsilon \to \hat{\gamma}$ in $C^2([0, \ell])$ as $\epsilon \to 0$. Hence
for some $\epsilon > 0$ small enough, we conclude that the curvature $\hat{\kappa}_\epsilon$ of
$\hat{\gamma}_\epsilon$ satisfies $\|\hat{\kappa}_\epsilon\|_{L^\infty(0, \ell)} \le \K_1(\gamma) \le 1/R$.
Hence we have found a shorter curve with the same boundary data satisfying the required curvature constraint.
\end{proof}

Now suppose that $n = 2$. We wish to give an alternative proof of Dubins's
main result \cite[Theorem I]{Dubins:57} based on Theorem \ref{thm:structure}.
Let $k > 0$ and consider a $1/k$-geodesic parametrised by
$\gamma \in \mathcal{G}$. Then Proposition \ref{prop:shortest_curves} and Theorem \ref{thm:structure}
imply that $\gamma$ is consistent with one of the descriptions \ref{item:arc-line-arc} or
\ref{item:arc-arc} in the introduction.

In the case \ref{item:arc-line-arc}, it is clear that any
minimiser of the length will not contain any full circles, so the curve will at most consist of a
circular arc, followed by a line segment, followed by another circular arc. This is one of the solutions described by Dubins.

In the case \ref{item:arc-arc}, we have a sequence of several circular arcs.
If there were more than four pieces, then it is also easy to see that a piece of
the curve could be replaced by a line segment, thus reducing the length. This is of
course impossible for a minimiser of the length, hence we have four or fewer pieces.
In order to see that four consecutive circular arcs are also impossible, we still need
Dubins's Lemma 2. Almost all of Dubins's other arguments, however, have been bypassed.

Sussmann's results for $n = 3$ \cite[Theorem 1]{Sussmann:95} follow in a similar way from Theorem \ref{thm:structure}
and again one of Dubins's lemmas. If we have a solution as in statement \ref{item:2D}, then we first
distinguish the following two cases. If the entire curve is planar, we apply the
above reasoning. (Sussmann's theorem contains another statement in this case, which is a consequence
of a result of Dubins \cite[Sublemma]{Dubins:57}.) Otherwise, we note that the curve must meet
the line $\mathcal{L}$ tangentially.
Then we may have a circular arc at either end of the curve and we may have some
intermediate pieces. But if one of these intermediate pieces is not a segment of $\mathcal{L}$,
it is clear that it must be a full circle. This clearly cannot happen for a solution of the
Markov-Dubins problem, so in fact we have (at most) a concatenation of a circular arc, a line,
and another circular arc. A solution as in statement \ref{item:3D}, on the other hand,
is a helicoidal arc in Sussmann's terminology.

\section{Examples} \label{sect:examples}

We finally examine a few examples of minimisers and $\infty$-elasticas,
which highlight some features and some limitations of the theory. Throughout
this section, we assume that $\alpha \equiv 1$.

\begin{example}[Circular arc] \label{ex:arc}
We first consider a circular arc parametrised by $\gamma \colon [0, \ell] \to \R^2$ with
$\gamma(s) = r(\cos(s/r), \sin(s/r))$ and with tangent vector $T(s) = (-\sin(s/r), \cos(s/r))$
and constant curvature $k = 1/r$. This is an $\infty$-elastica by Theorem \ref{thm:structure}.
If we want to check equations \eqref{eqn:ODE1} and \eqref{eqn:ODE2} directly, then we first compute $T'' + k^2 T = 0$. Moreover, the
vectors $T$ and $T'$ span $\R^2$ everywhere, so $\proj_{T, T'}^\perp(\lambda) = 0$
regardless of the value of $\lambda$. Thus we only need to consider equation
\eqref{eqn:ODE2}, which gives $g'(s) = -\frac{1}{r}(\lambda_1 \cos(s/r) + \lambda_2 \sin(s/r))$.
This is satisfied for $g(s) = \lambda_2 \cos(s/r) - \lambda_1 \sin(s/r) + h = \lambda \cdot T(s) + h$
for any $h \in \R$. Clearly we can choose $h$ such that $g \ge 0$ in $[0, \ell]$.

Now suppose that we wish to apply Theorem \ref{thm:minimiser}.
We have a minimiser of $\K_1$ if the inequalities 
$0 \le \lambda \cdot T + h \le  -\lambda \cdot T$
are satisfied simultaneously. They give rise to the conditions
\[
\frac{h}{2} \le \min_{[0, \ell]} (-\lambda \cdot T) \le \max_{[0, \ell]} (-\lambda \cdot T) \le h.
\]
It is possible to satisfy these if, and only if, $\ell \le 2\pi r/3$,
in which case we can choose $\lambda = (\sqrt{3}/2, -1/2)$ and $h = 1$.
Thus a circular arc of radius $r$ minimises $\K_1$ if its length does not exceed $2\pi r/3$.
\end{example}

The example shows that the condition of Theorem \ref{thm:minimiser} is sufficient but not necessary, for the
above circular arc is still a minimiser as long as $\ell \le 2\pi r$ by the results of
Schmidt \cite{Schmidt:25}.

Next we consider the question whether the notion of an $\infty$-elastica is genuinely
more general than that of a minimiser of $\K_\alpha$. The answer is yes, and the
following example gives a one-parameter family of $\infty$-elasticas that are not minimisers
and not even local minimisers with respect to the $W^{1, 2}$-topology.

\begin{example}[Non-minimising $\infty$-elastica] \label{ex:non-minimising}
Consider curves with end points $a_1 = (-1, 0)$ and $a_2 = (1, 0)$ and tangent vectors
$T_1 = (0, 1)$ and $T_2 = (0, -1)$. If $\ell = \pi$, then there is one candidate that
consists of three semicircles of radius $1/3$; this is illustrated in Figure \ref{fig:non-min}.
It is an $\infty$-elastica by Theorem \ref{thm:structure}.
\begin{figure}[h!tb]
\centering
\begin{subfigure}[t]{0.45\textwidth}
\includegraphics[width=\textwidth]{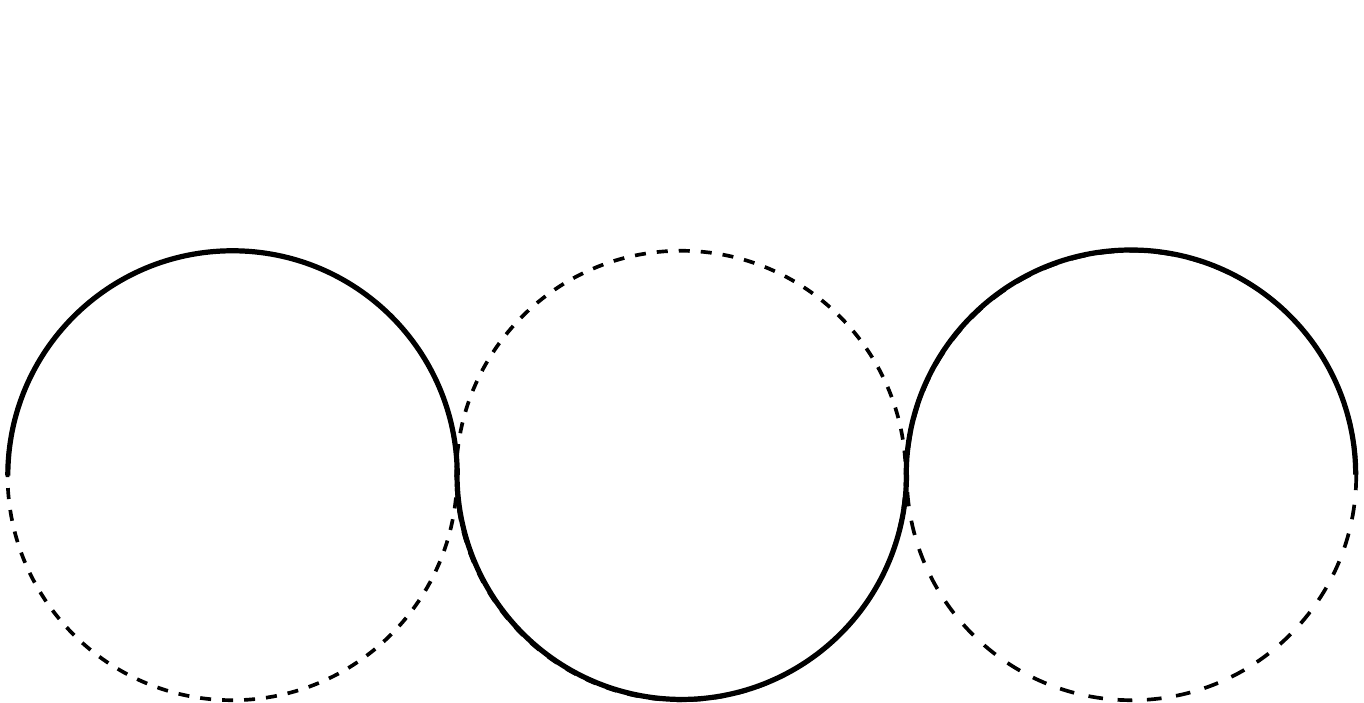}
\caption{The $\infty$-elastica}
\label{fig:non-min}
\end{subfigure}
\hfill
\begin{subfigure}[t]{0.45\textwidth}
\includegraphics[width=\textwidth]{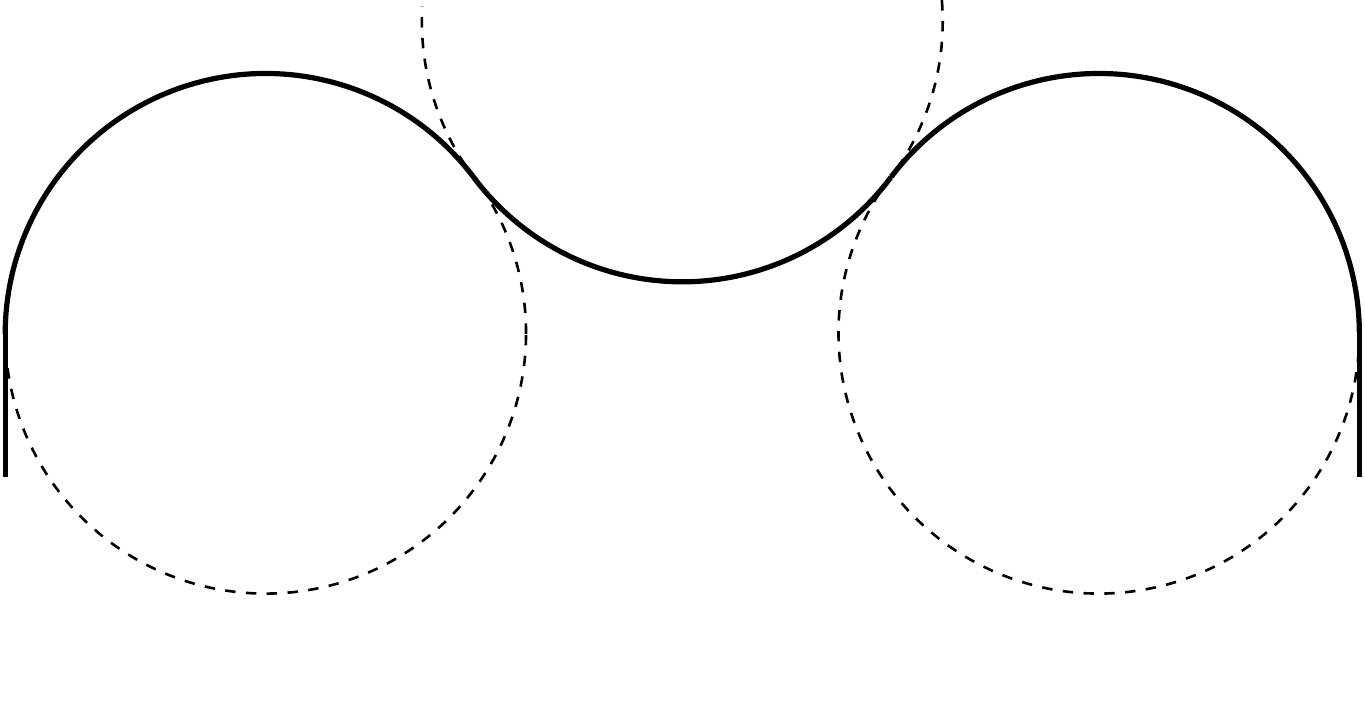}
\caption{A comparison curve}
\label{fig:comparison}
\end{subfigure}
\caption{Construction of an $\infty$-elastica that is not a minimiser}
\end{figure}

For $r \in [1/3, 1)$, we also construct some comparison curves including three circular arcs of radius $r$.
To this end, define $\omega(r) = \arccos((1 - r)/2r)$.
For $h \in \R$, there is a curve comprising three circular arcs of radius $r$, with centres
\[
(r - 1, h), \quad (0, h + 2r\sin \omega(r)), \quad (1 - r, h),
\]
that connects the points $(-1, h)$ and $(1, h)$. The length of this curve is $\tilde{\ell}(r) = r(3\pi - 4\omega(r))$.
We compute $\tilde{\ell}(1/3) = \pi = \tilde{\ell}(1)$ and
\[
\tilde{\ell}''(r) = \frac{4(1 - r)}{r(3r^2 + 2r - 1)^{3/2}} > 0
\]
in $(1/3, 1)$. Hence $\tilde{\ell}(r)< \pi$ for all $r \in (1/3, 1)$.
If we choose $h = (\pi - \tilde{\ell}(r))/2$, we can attach a line segment to each end
and thereby construct a comparison curve of length $\pi$ that satisfies the required boundary conditions
(see Figure \ref{fig:comparison}). But the value of $\K_1$ is $1/r < 3$.
\end{example}

Finally we have an example of a three-dimensional $\infty$-elastica,
showing that both cases in Theorem \ref{thm:structure} can indeed occur.

\begin{example}[Helical arc] \label{ex:helix}
Consider $\gamma \colon [0, \ell] \to \R^3$ given by
\[
\gamma(s) = (r\cos \omega \cos(s/r), r\cos \omega \sin(s/r), s\sin \omega)
\]
for some $\omega \in (0, \pi/2)$. The curvature of this curve is $k = r^{-1} \cos \omega$.
For $T = \gamma'$, we compute
\[
T'' + k^2 T = \frac{\sin \omega \cos \omega}{r^2}(\sin \omega \sin (s/r), - \sin \omega \cos(s/r), \cos \omega).
\]
Now let $\lambda = (0, 0, 1)$. Then $\lambda \cdot T = \sin \omega$. In order to find
$\proj_{T, T'}^\perp(\lambda)$, we first compute
\[
N = \frac{r}{\cos \omega} T \times T' = (\sin \omega \sin(s/r), -\sin \omega \cos(s/r), \cos \omega)
\]
and note that $N$ is a unit vector perpendicular to $T$ and $T'$. Hence
\[
\proj_{T, T'}(\lambda) = (\lambda \cdot N)N = \cos \omega (\sin \omega \sin(s/r), -\sin \omega \cos(s/r), \cos \omega).
\]
Choosing $\eta = \sin \omega - \cos \omega \cot \omega$, we see that equation \eqref{eqn:ODE_alpha=1}
is satisfied. Hence $\gamma$ is an $\infty$-elastica.
\end{example}

\def\cprime{$'$}
\providecommand{\bysame}{\leavevmode\hbox to3em{\hrulefill}\thinspace}
\providecommand{\MR}{\relax\ifhmode\unskip\space\fi MR }
\providecommand{\MRhref}[2]{%
  \href{http://www.ams.org/mathscinet-getitem?mr=#1}{#2}
}
\providecommand{\href}[2]{#2}


\begin{thebibliography}{10}

\bibitem{Aronsson:65}
G.~Aronsson, \emph{Minimization problems for the functional {${\rm
  sup}_{x}\,F(x,\,f(x),\,f^{\prime} (x))$}}, Ark. Mat. \textbf{6} (1965),
  33--53.

\bibitem{Aronsson:66}
\bysame, \emph{Minimization problems for the functional {${\rm sup}_{x}\, F(x,
  f(x),f\sp\prime (x))$}. {II}}, Ark. Mat. \textbf{6} (1966), 409--431.

\bibitem{Aronsson:67}
\bysame, \emph{Extension of functions satisfying {L}ipschitz conditions}, Ark.
  Mat. \textbf{6} (1967), 551--561.

\bibitem{Aronsson:10}
\bysame, \emph{On certain minimax problems and {P}ontryagin's maximum
  principle}, Calc. Var. Partial Differential Equations \textbf{37} (2010),
  no.~1-2, 99--109.

\bibitem{Blaschke:45}
W.~Blaschke, \emph{Vorlesungen \"uber {D}ifferentialgeometrie und geometrische
  {G}rundlagen von {E}insteins {R}elativit\"atstheorie. {B}and {I}.
  {E}lementare {D}ifferentialgeometrie}, Dover Publications, New York, N. Y.,
  1945, 3d ed.

\bibitem{Bryant-Griffiths:86}
R.~Bryant and P.~Griffiths, \emph{Reduction for constrained variational
  problems and {$\int{1\over 2}k^2\,ds$}}, Amer. J. Math. \textbf{108} (1986),
  no.~3, 525--570.

\bibitem{Chern:67}
S.~S. Chern, \emph{Curves and surfaces in {E}uclidean space}, Studies in
  {G}lobal {G}eometry and {A}nalysis, Math. Assoc. Amer. (distributed by
  Prentice-Hall, Englewood Cliffs, N.J.), 1967, pp.~16--56.

\bibitem{Dubins:57}
L.~E. Dubins, \emph{On curves of minimal length with a constraint on average
  curvature, and with prescribed initial and terminal positions and tangents},
  Amer. J. Math. \textbf{79} (1957), 497--516.

\bibitem{Euler:1744}
L.~Euler, \emph{{M}ethodus inveniendi lineas curvas maximi minimive proprietate
  gaudentes sive solutio problematis isoperimetrici latissimo sensu accepti},
  Marc-MichelBousquet \& Co., Lausanne--Geneva, 1744.

\bibitem{Ferone-Kawohl-Nitsch:18}
V.~Ferone, B.~Kawohl, and C.~Nitsch, \emph{Generalized elastica problems under
  area constraint}, Math. Res. Lett. \textbf{25} (2018), no.~2, 521--533.

\bibitem{Huang:04}
R.~Huang, \emph{A note on the {$p$}-elastica in a constant sectional curvature
  manifold}, J. Geom. Phys. \textbf{49} (2004), no.~3-4, 343--349.

\bibitem{Katzourakis:15}
N.~Katzourakis, \emph{An introduction to viscosity solutions for fully
  nonlinear {PDE} with applications to calculus of variations in {$L^\infty$}},
  SpringerBriefs in Mathematics, Springer, Cham, 2015.

\bibitem{Katzourakis-Moser:19}
N.~Katzourakis and R.~Moser, \emph{Existence, uniqueness and structure of
  second order absolute minimisers}, Arch. Ration. Mech. Anal. \textbf{231}
  (2019), no.~3, 1615--1634.

\bibitem{Katzourakis-Parini:17}
N.~Katzourakis and E.~Parini, \emph{The eigenvalue problem for the {$\infty
  $}-{B}ilaplacian}, NoDEA Nonlinear Differential Equations Appl. \textbf{24}
  (2017), no.~6, Art. 68, 25.

\bibitem{Katzourakis-Pryer:18}
N.~Katzourakis and T.~Pryer, \emph{Second-order {$L^\infty$} variational
  problems and the $\infty$-polylaplacian}, Adv. Calc. Var. (Ahead of Print)
  (2018), DOI:10.1515/acv-2016-0052.

\bibitem{Katzourakis-Pryer:19}
\bysame, \emph{On the numerical approximation of {$p$}-biharmonic and
  {$\infty$}-biharmonic functions}, Numer. Methods Partial Differential
  Equations \textbf{35} (2019), no.~1, 155--180.

\bibitem{Langer-Singer:84_2}
J.~Langer and D.~A. Singer, \emph{Knotted elastic curves in {${\bf R}^3$}}, J.
  London Math. Soc. (2) \textbf{30} (1984), no.~3, 512--520.

\bibitem{Langer-Singer:84_1}
\bysame, \emph{The total squared curvature of closed curves}, J. Differential
  Geom. \textbf{20} (1984), no.~1, 1--22.

\bibitem{Laumond-Sekhavat-Lamiraux:98}
J.~Laumond, S.~Sekhavat, and F.~Lamiraux, \emph{Guidelines in nonholonomic
  motion planning for mobile robots}, Robot Motion Planning and Control,
  Lecture Notes in Control and Information Sciences, vol. 229, Springer-Verlag,
  Berlin--Heidelberg, 1998, pp.~1--53.

\bibitem{Linner:98}
A.~Linn\'{e}r, \emph{Explicit elastic curves}, Ann. Global Anal. Geom.
  \textbf{16} (1998), no.~5, 445--475.

\bibitem{Markov:1887}
A.~A. Markov, \emph{Some examples of the solution of a special kind of problem
  on greatest and least quantities}, Soobshch. Kharkovsk. Mat. Obshch.
  \textbf{1} (1887), 250--276, in Russian.

\bibitem{Masnou-Morel:98}
S.~Masnou and J.-M. Morel, \emph{Level lines based disocclusion}, Proceedings
  1998 IEEE International Conference on Image Processing, 1998, pp.~259--263.

\bibitem{Moser-Schwetlick:12}
R.~Moser and H.~Schwetlick, \emph{Minimizers of a weighted maximum of the
  {G}auss curvature}, Ann. Global Anal. Geom. \textbf{41} (2012), no.~2,
  199--207.

\bibitem{Oldfather-Ellis-Brown:33}
W.~A Oldfather, C.~A. Ellis, and D.~M. Brown, \emph{{L}eonhard {E}uler's
  elastic curves}, Isis \textbf{20} (1933), 72--160.

\bibitem{Reeds-Shepp:90}
J.~A. Reeds and L.~A. Shepp, \emph{Optimal paths for a car that goes both
  forwards and backwards}, Pacific J. Math. \textbf{145} (1990), no.~2,
  367--393.

\bibitem{Sakellaris:17}
Z.~N. Sakellaris, \emph{Minimization of scalar curvature in conformal
  geometry}, Ann. Global Anal. Geom. \textbf{51} (2017), no.~1, 73--89.

\bibitem{Schmidt:25}
E.~Schmidt, \emph{\"{U}ber das {E}xtremum der {B}ogenl\"ange einer {R}aumkurve
  bei vergeschriebenen {E}inschr\"ankungen ihrer {K}r\"ummung},
  Sitzungsberichte preuss. Akad. Wissenschaften Berlin \textbf{1925} (1925),
  485--490.

\bibitem{Schur:21}
A.~Schur, \emph{\"{U}ber die {S}chwarzsche {E}xtremaleigenschaft des {K}reises
  unter den {K}urven konstanter {K}r\"ummung}, Math. Ann. \textbf{83} (1921),
  143--148.

\bibitem{Simon:96}
L.~Simon, \emph{Theorems on regularity and singularity of energy minimizing
  maps}, Lectures in Math.~ETH Z\"urich, Birkh\"auser, Basel, 1996.

\bibitem{Sussmann:95}
H.~J. Sussmann, \emph{Shortest 3-dimensional paths with a prescribed curvature
  bound}, Proceedings of 1995 34th IEEE Conference on Decision and Control,
  vol.~4, 1995, pp.~3306--3312.

\end{thebibliography}
\end{document}